\documentclass{amsart}
\usepackage{amsmath,amssymb,amsthm}
\newtheorem{theo}{Theorem}
\newtheorem{defi}[theo]{Definition}
\newtheorem{lemm}[theo]{Lemma}
\newtheorem{prop}[theo]{Proposition}
\newtheorem{cor}[theo]{Corollary}
\newtheorem{conj}[theo]{Conjecture}

\newtheorem*{maintheo}{Main Theorem}

\theoremstyle{definition}
\newtheorem{rema}[theo]{Remark}
\newtheorem{exam}[theo]{Example}
\begin{document}

\title{A Classification of Flows on AFD Factors with Faithful Connes--Takesaki Modules}
\author{Koichi Shimada}
\email{shimada@ms.u-tokyo.ac.jp}
\address{Department of Mathematical Sciences
University of Tokyo, Komaba, Tokyo, 153-8914, Japan}
\date{}
\begin{abstract}
We completely classify flows on approximately finite dimensional (AFD) factors with faithful Connes--Takesaki modules up to cocycle conjugacy. This is a generalization of the uniqueness of the trace-scaling flow on the AFD factor of type $\mathrm{II}_\infty $, which is equivalent to the uniqueness of the AFD factor of type $\mathrm{III}_1$. In order to achieve this, we show that a flow on any AFD factor with faithful Connes--Takesaki module has the Rohlin property, which is a kind of outerness for flows introduced by Kishimoto and Kawamuro. 
\end{abstract}

\maketitle
\section{Introduction}
\label{intro}
In 1987, Haagerup \cite{H} showed the uniqueness of the approximately finite dimensional (AFD) factor of type $\mathrm{III}_1$. Although this is a great theorem which was the final step of the classification of AFD factors, here, we think of this theorem as a part of theory of flows on von Neumann algebras. In fact, by Takesaki's duality theorem, the uniqueness of the AFD factor of type $\mathrm{III}_1$ is equivalent to the uniqueness of the trace-scaling flow on the AFD factor of type $\mathrm{II}_\infty $. In this paper, we will generalize this result as a part of theory of flows. More precisely, we will show the following theorem.

\begin{maintheo}
\textup{(Theorem \ref {main})} Let $M$ be any AFD factor. Then flows on $M$ which are not approximately inner at any non-trivial point \textup{(}or equivalently, have faithful Connes--Takesaki modules\textup{)} are classified by their Connes--Takesaki modules, up to \textup{(}strong\textup{)} cocycle conjugacy.
\end{maintheo}

 This result is not only a generalization of the uniqueness of the AFD factor of type $\mathrm{III}_1$. It is related to an important problem of classification of flows. The problem is about an ``outerness'' of flows, the Rohlin property.

Now, we explain the Rohlin property. First of all, it is important to note that classification of flows is difficult, compared with the complete classification of actions of discrete amenable groups on AFD factors (See Connes \cite{C3}, Jones \cite{J}, Ocneanu \cite{O}, Sutherland--Takesaki \cite{ST2}, Kawahigashi--Sutherland--Takesaki \cite{KwhST} and Katayama--Sutherland--Takesaki \cite{KtST}). The difficulty seems to come from the difference among various ``outerness conditions'' of flows. For example, one might consider that flows which are outer at any non-trivial point are outer. It seems to be reasonable to think of flows with full Connes spectra to be outer. However, the problem is that these outernesses do not coincide (See Example 2.3 of Kawamuro \cite{Kwm}, which is based on Kawahigashi \cite{Kwh1}, \cite{Kwh2}). Thus we need to clarify what appropriate outerness is. As a candidate for appropriate outerness, the Rohlin property was introduced by Kishimoto \cite{Ksm}. The Rohlin property is an analogue of the property in non-commutative Rohlin's lemma in Connes' classification of actions of $\mathbf{Z}$ (Theorem 1.2.5 of Connes \cite{C3}), which is derived from outerness of actions of $\mathbf{Z}$. Actually, Kishimoto's definition is for flows on $\mathrm{C}^*$-algebras. After Kishimoto's work, Kawamuro \cite{Kwm} introduced the Rohlin property for flows on von Neumann algebras. Recently, Masuda--Tomatsu \cite{MT} have presented a classification theorem for Rohlin flows. Thus the Rohlin property is now considered to be appropriate outerness. However, there is a problem. In general, it is not easy to see whether a given flow has the Rohlin property or not. Moreover, the Rohlin property is not written by ``standard invariants'' for flows. This can be an obstruction for the complete classification of flows on AFD factors. Hence it is important to characterize the Rohlin property in an appropriate way. At this point, it is conjectured that a flow on an AFD factor has the Rohlin property if and only if it has full Connes spectrum and is centrally free at each non-trivial point.

Now, we explain the relation between our main theorem and this characterization program. First of all, the uniqueness of the trace-scaling flow on the AFD factor of type $\mathrm{II}_\infty$ is deeply related to its having the Rohlin property. Indeed, by the results of Connes \cite{C2} and Haagerup \cite{H}, it is possible to see that any trace-scaling flow on the AFD factor of type $\mathrm{II}_\infty$ has the Rohlin property (See Theorem 6.18 of Masuda--Tomatsu \cite{MT}). The uniqueness follows from the classification theorem of Rohlin flows. Thus it is expected that flows have the Rohlin property under our generalized assumption, that is, having faithful Connes--Takesaki modules (See Problem 8.5 of Masuda--Tomatsu \cite{MT}). In this paper, we actually show that flows on any AFD factor with faithful Connes--Takesaki modules have the Rohlin property, and obtain the main theorem by using Masuda--Tomatsu's theorem. Hence it is possible to think of our main theorem as a partial answer to the characterization problem of the Rohlin property. The theorem means that if a flow is ``very outer'' at any non-trivial point, then it is globally ``very outer''. Our main theorem provides interesting examples of Rohlin flows, and we believe that it is a useful observation for the characterization problem.

Hence the difficult point of the proof of the main theorem is to show the Rohlin property. In order to show the Rohlin property of a flow $\alpha$ on a factor $M$, we need to find good unitaries of $M$. To achieve this, we consider the continuous decomposition of $M$. The dual action $\theta $ of a modular flow of $M$ and the canonical extention $\tilde{\alpha }$ of $\alpha $ act on the continuous core $\tilde{M}$ of $M$. If the action $\theta \circ \tilde{\alpha }$ of $\mathbf{R}^2$ is faithful on the center of $\tilde{M}$, then the Rohlin property of $\alpha $ follows from ergodic theory. The problem is that even if $\alpha $ has faithful Connes--Takesaki module, the restriction of $\tilde{\alpha }\circ \theta $ on the center of $\tilde{M}$ may NOT be faithful. In order to overcome this problem, we consider a kind of decomposition of actions over the center of $\tilde {M}$ and reduce the problem to trace-scaling actions of $\mathbf{Z}$, $\mathbf{Z}^2$ or $\mathbf{R}$ on the AFD factor of type $\mathrm{II}_\infty$.  

Besides our theorem, there is a similar result about actions of compact groups due to Izumi \cite{I}. He has shown that if an action of a compact group on a factor of type $\mathrm{III}$ has faithful Connes--Takesaki module, then it is minimal. It is also possible to consider that his result means that ``pointwise outerness'' implies ``global outerness''. Although there is similarity between our main theorem and his theorem, however, there are some observations which show that our main theorem is essentially different from his theorem. For example, for actions of compact groups with faithful Connes--Takesaki modules, cocycle conjugacy always implies conjugacy. This is also true for trace-scaling flows on factors of type $\mathrm{II}_\infty$. However, for our flows, this does not hold.  

Finally, we note that our main theorem depends on the uniqueness of the trace-scaling flow, which is based on the results of Connes \cite{C2} and Haagerup \cite{H}. Showing the uniqueness of the trace-scaling flow on the AFD factor of type $\mathrm{II}_\infty$ without using Connes  and Haagerup's theory is also an important problem (See Problem 8.8 of Masuda--Tomatsu \cite{MT}). However, our main theorem is different from that problem.

\section{Preliminaries}
First of all, we explain things which are important to understand our main theorem and its proof, that is, Connes--Takesaki module and the Rohlin property. 
\subsection{Connes--Takesaki module}
First of all, we recall Connes--Takesaki module. Basic references are Connes--Takesaki \cite{CT} and Haagerup--St\o rmer \cite{HS}. 

Let $M$ be a properly infinite factor and let $\phi $ be a normal faithful semifinite weight on $M$. Set $N:=M\rtimes _{\sigma ^\phi}\mathbf{R}$. Then the von Neumann algebra $N$ is generated by $M$ and a one parameter unitary group $\{ \lambda _s\}_{s\in \mathbf{R}}$ satisfying $\lambda _s x \lambda _{-s}= \sigma _s^\phi (x)$ for $x\in M$, $s\in \mathbf{R}$. Let $\theta ^\phi$ be the dual action of $\sigma^\phi $ and let $C$ be the center of $N$. Then an automorphism $\alpha $ of $M$ extends to an automrphism $\tilde{\alpha }$ of $N$ by the following way (See Proposition 12.1 of Haagerup--St\o mer \cite{HS}).

\[ \tilde{\alpha }(x)=\alpha (x)\ \mathrm{for}\ x\in M,\ \tilde{\alpha }({\lambda _s})=[D\phi \circ \alpha ^{-1}: D\phi]_s\lambda _s\ \mathrm{for}\ s\in \mathbf{R}. \]

This $\tilde{\alpha }$ has the following properties (See Proposition 12.2 of Haagerup--St\o mer \cite{HS}).

\bigskip

(1) The automorphism $\tilde{\alpha }$ commutes with $\theta ^\phi$.

(2) The automorphism $\tilde{\alpha }$ preserves the canonical trace on $N$.

(3) The map $\alpha \mapsto \tilde {\alpha }$ is a continuous group homomorphism.

\bigskip 

Set $\mathrm{mod}^\phi(\alpha ):=\tilde{\alpha }\mid _C$. This is said to be a Connes--Takesaki module of $\alpha$. Actually, this definition is different from the original definition of Connes--Takesaki \cite{CT}. However, in Proposition 13.1 of Haagerup--St\o mer \cite{HS}, it is shown that they are same. This Connes--Takesaki module does not depend on the choice of $\phi$, that is, if $\phi $ and $\psi $ are two normal faithful semifinite weights, then the action $\mathrm{mod}^\phi (\alpha ) \circ \theta ^\phi$ of $\mathbf{R}\times \mathbf{Z}$ on $C$ is conjugate to $\mathrm{mod}^\psi (\alpha ) \circ \theta ^\psi $.  Hence, in the following, we will omit $\phi$, and write $\theta _t$ and $\mathrm{mod}(\alpha )$ if there is no danger of confusion. For an automorphism of any factor of type $\mathrm{II}_\infty$, considering its Connes--Takesaki module is equivalent to considering how it scales the trace. Hence flows with faithful Connes--Takesaki modules are natural generalization of trace-scaling flows.

We explain what property of automorphisms Connes--Takesaki module indicates. By Theorem 1 of Kawahigashi--Sutherland--Takesaki \cite{KwhST}, an automorphism of any AFD factor is approximately inner if and only if its Connes--Takesaki module is trivial. Hence Connes--Takesaki module indicates ``the degree of approximate innerness''. 

\subsection{Rohlin flows}
 Next, we recall the Rohlin property, which is a kind of ``outerness'' for flows. A basic reference is Masuda--Tomatsu \cite{MT}. One of the typical forms of classification theorems of group actions is the following.

\bigskip

 Two ``very outer'' actions whose difference are approximately inner are cocycle conjugate. 

\bigskip

Hence, in order to classify flows on AFD factors, we need to clarify what is appropriate outerness. However, the problem is not so simple, compared with the problem for discrete group actions. As a candidate for appropriate outerness, the Rohlin property was introduced by Kishimoto \cite{Ksm} and Kawamuro \cite{Kwm}. In the following, we will explain the definition.

 Let $\omega $ be a free ultrafilter on $\textbf{N} $. We denote by $l^{\infty }(M)$ the $\mathrm{C}^*$-algebra which consists of all bouded sequences in $M$. Set
\[ I_{\omega }:=\{ (x_n) \in l^{\infty }(M) \mid \text{strong*-lim}_{n\to \omega }(x_n)=0 \}, \]
\begin{align*}
 N_{\omega }:=\{ (x_n) \in l^{\infty }(M) \mid &\text{for all } (y_n)\in I_{\omega }, \\
                                            &\text{ we have } (x_ny_n)\in I_\omega \text{ and }(y_nx_n)\in I_\omega \} ,
 \end{align*}
\[ C_{\omega }:=\{ (x_n) \in l^{\infty }(M) \mid \text{for all } \phi \in M_{*}, \lim _{n\to \omega }\| [\phi ,x_n] \| =0 \}. \]
Then we have $I_{\omega }\subset C_{\omega }\subset N_{\omega }$ and $I_{\omega }$ is a closed ideal of $N_{\omega}$. Hence we define the quotient $\mathrm{C}^{*}$-algebra $M^{\omega }:=N_{\omega }/I_{\omega }$. Denote the canonical quotient map $N_{\omega }\to M^{\omega }$ by $\pi $. Set $M_{\omega }:=\pi (C_{\omega })$. Then $M_{\omega }$ and $M^{\omega }$ are von Neumann algebras as in Proposition 5.1 of Ocneanu \cite{O}. Let $\alpha $ be an automorphism of $M$. We define an automorphism $\alpha ^{\omega }$ of $M^{\omega }$ by $\alpha^{\omega }((x_n))=(\alpha (x_n))$ for $(x_n)\in M^{\omega}$. Then we have $\alpha^{\omega }(M_{\omega })=M_{\omega }$. By restricting $\alpha ^{\omega }$ to $M_{\omega }$, we define an automorphism $\alpha _{\omega }$ of $M_{\omega }$. Hereafter we denote $\alpha ^{\omega }$ and $\alpha _{\omega }$ by $\alpha $ if there is no danger of confusion. 

 Choose a normal faithful state $\varphi $ on $M$. For a flow $\alpha $ on a von Neumann algebra $M$, set 
\begin{multline*}
 M_{\omega ,\alpha }:=\{ (x_n)\in M_\omega \mid \text{for all } \epsilon >0,\text{ there exists } \delta >0 \text{ such that } \\
      \{n\in \textbf{N}\mid \| \alpha _t(x_n)-x_n\|_{\varphi }^{\sharp } <\epsilon   \ \mathrm{ for }\ |t|<\delta \}\in \omega \}. 
\end{multline*}

\begin{defi} 
\textup{(See Definition 4.1 of Masuda--Tomatsu \cite{MT})} A flow $\alpha$ on a factor $M$ is said to have the Rohlin property if for each $p\in \mathbf{R}$, there exists a unitary $u$ of $M_{\omega, \alpha }$ satisfying $\alpha _t(u)=e^{-ipt} u$ for all $t\in \mathbf{R}$.
\end{defi}
For flows with the Rohlin property, there is a classification theorem due to Masuda--Tomatsu \cite{MT}.
\begin{theo}
\textup{(See Theorem 5.14 of Masuda--Tomatsu \cite{MT})} Let $\alpha ^1$, $\alpha ^2$ be two Rohlin flows on a separable von Neumann algebra $M$. If $\alpha ^1_t\circ \alpha ^2_{-t}$ is approximated by inner automorphisms for each $t\in \mathbf{R}$, then they are mutually \textup{(}strongly\textup{)} cocycle conjugate \textup{(}For the definition of strong cocycle conjugacy, see Subsection 2.2 of Masuda--Tomatsu \cite{MT}\textup{)}. \label{Masuda--Tomatsu}
\end{theo}
Hence, the Rohlin property is now considered to be appropriate outerness. However, one may feel that the definition of the Rohlin property is not so simple. Hence characterization of the Rohlin property is an important problem (See Conjecture 8.3 of Masuda--Tomatsu \cite{MT}). 

\section{Main Results}
The main theorem of this paper is the following.
\begin{theo}
\label{main}
Flows on any AFD factor with faithful Connes--Takesaki modules are completely classified by their Connes--Takesaki modules, up to strong cocycle conjugacy.
\end{theo}
As a corollary, we obtain a classification theorem up to cocycle conjugacy. For a von Neumann algebra $C$ and a flow $\beta $ of $C$, set $\mathrm{Aut}_\beta (C):=\{ \sigma \in \mathrm{Aut}(C)\mid \sigma \circ \beta _t=\beta _t\circ \sigma, \ t\in \mathbf{R} \}$.
\begin{cor}
Let $\alpha ^1$ and $\alpha ^2$ be two flows on an AFD factor $M$ with faithful Connes--Takesaki modules. Then they are cocycle conjugate if and only if there exists an automorphism $\sigma \in \mathrm{Aut}_\theta (C)$ with $\mathrm{mod}(\alpha ^2_t)=\sigma \circ \mathrm{mod}(\alpha ^1_t) \circ \sigma ^{-1}$ for any $t\in \mathbf{R}$. 
\end{cor}
As an obvious application, we have the following example.
\begin{exam}
A flow on any AFD factor with faithful Connes--Takesaki module absorbs any flow on the AFD $\mathrm{II}_1$ factor, as a tensor product factor.
\end{exam}
In order to show Theorem \ref{main}, by Theorem \ref{Masuda--Tomatsu}, and the characterization of approximate innerness of automorphisms of AFD factors (Theorem 1 of Kawahigashi--Sutherland--Takesaki \cite{KwhST}), it is enough to show the following theorem. 
\begin{theo}
\label{Rohlin}
A flow on any AFD factor with faithful Connes--Takesaki module has the Rohlin property.
\end{theo}
From what we have explained in the previous section, this theorem means that a kind of ``pointwise outerness'' implies ``global outerness''.

\bigskip

As we have explained in Section \ref{intro} and the previous subsection, characterization of the Rohlin property is an important problem (Conjecture 8.3 of Masuda--Tomatsu \cite{MT}). Theorem \ref{Rohlin} gives a partial answer to this problem.  We will proceed further to this direction in Subsection \ref{generalization}. 

\section{The Proof of Main Results}

In this section, we show Theorem \ref{Rohlin}. In order to achieve this, we first note that we may assume that a flow has an invariant weight. This is seen in the following way. Let $\alpha $ be a flow on an AFD factor $M$. Then by the same argument as in Lemma 5.10 of Sutherland--Takesaki \cite{ST2} (or equivalently, by the combination of Lemma 5.11 and Lemma 5.12 of \cite{ST2}), there exists a flow $\beta $ and a dominant weight $\phi $ which satisfy the following conditions.

\bigskip

(1) We have $\phi \circ \beta _t =\phi $ for all $t\in \mathbf{R}$.

(2) The action $\beta $ is cocycle conjugate to $\alpha \otimes \mathrm{id}_{B(L^2\mathbf{R})}$.

\bigskip

 By Lemma 2.11 of Connes \cite{C}, $(M\otimes B(L^2\mathbf{R}))_\omega =M_\omega \otimes \mathbf{C}$. Hence, by replacing $\alpha $ by $\beta $, we may assume that the action $\alpha $ has an invariant dominant weight. In the rest of the paper, we denote the continuous core $M\rtimes _{\sigma ^\phi} \mathbf{R}$ by $N$ and the dual action of $\sigma ^\phi $ by $\theta$. Then by the same argument as in the proof of Proposition 13.1 of Haagerup--St\o rmer \cite{HS}, the action $\tilde {\alpha }$ extends to a flow $\tilde{\tilde{\alpha }}$ of $N\rtimes _\theta \mathbf{R}$ so that if we identify $N\rtimes _\theta \mathbf{R}$ with $M\otimes B(L^2\mathbf{R})$ by Takesaki's duality, $\tilde{\tilde{\alpha }}$ corresponds to $\alpha \otimes \mathrm{id}$.  By Lemma 2.11 of Connes \cite{C} again, in order to show that $\alpha$ has the Rohlin property, it is enough to show that $\tilde{\tilde{\alpha }}$ has the Rohlin property. In order to achieve this, we need to choose $\{u_n \}\subset U(M\otimes B(L^2\mathbf{R}))_\omega $ which satisfies the conditions in the definition of the Rohlin property. Our strategy is to choose $\{u_n \}$ from $N$. Based on this strategy, it is sufficient to show the following lemma.

\begin{lemm}
For each $p\in \mathbf{R}$, there exists a sequence $\{u_n\} \subset U(N)$ satisfying the following conditions.

\textup{(1)} We have $\| [u_n, \phi ]\| \to 0$ for any $\phi \in N_{*}$.

\textup{(2)} We have $\theta _s(u_n)-u_n \to 0$ compact uniformly for $s\in \mathbf{R}$ in the strong* topology.

\textup{(3)} We have $\tilde{\alpha}_t (u_n)-e^{ipt}u_n \to 0$ compact uniformly for $t\in \mathbf{R}$ in the strong* topology.
\label{key}
\end{lemm} 
By the first two conditions, this $\{ u_n\}$ asymptotically commutes with elements in a dense subspace of $M\otimes B(L^2\mathbf{R}) \cong M$. However, in general, this does not imply that $\{ u_n \}$ is centralizing (and this sometimes causes a serious problem). Hence, in  order to assure that Lemma \ref{key} implies Theorem \ref{Rohlin}, we need to show the following lemma.
\begin{lemm}
\label{7}
Let $M$ be an AFD factor of type $\mathrm{III}$ and let $M=N\rtimes _{\theta}\mathbf{R}$ be the continuous decomposition. Then a sequence $\{u_n\}\subset U(N)$ with conditions \textup{(1)} and \textup{(2)} of the above lemma is centralizing.
\end{lemm}
\begin{proof}
Let $H$ be the standard Hilbert space of $N$. Take $\xi \in H$ and $f\in L^2(\mathbf{R})$. Since 
\[ x(\xi \otimes f)(s) = (\theta _{-s}(x)\xi )f(s),\]
\begin{align*}
(\xi \otimes f)x(s)&=(J_Mx^*J_M(\xi \otimes f))(s) \\
                   &=(J_Nx^*J_N\xi )f(s) \\
                   &=(\xi x)f(s)
\end{align*}
 for $s\in \mathbf{R}$, $x\in N$, we have
\begin{align*}
\| u_n(\xi \otimes f) - (\xi \otimes f)u_n \| ^2 &=\int _{\mathbf{R}}\| \theta _{-s}(u_n)\xi -\xi u_n\| ^2| f(s)| ^2 \ ds\\
                                                               &\leq \int _{\mathbf{R}} \| (\theta _{-s}(u_n) -u_n )\xi \| ^2 | f(s) | ^2\ ds \\
                                                               & + \int _{\mathbf{R}}\| u_n \xi -\xi u_n \| ^2| f(s)| ^2 \ ds \\
                                                               &\to 0
\end{align*}
by Lebesgue's convergence theorem. Here, the convergence of the second term follows from Lemma 2.6 of Masuda--Tomatsu \cite{MT}. Any vector of $H\otimes L^2\mathbf{R}$ is approximated by finite sums of vectors of the form $\xi \otimes f$. Hence for any vector $\eta \in H\otimes L^2\mathbf{R}$, we have $\| u_n\eta -\eta u_n \| \to 0$. Hence $\{u_n \}$ is centralizing.
\end{proof}
 By this lemma, Lemma \ref{key} implies Theorem \ref{Rohlin}. In the following, we will show Lemma \ref{key}. If $M$ is of type $\mathrm{II}_\infty$, Lemma \ref{key} is shown in Theorem 6.18 of Masuda--Tomatsu \cite{MT}, using Connes and Haagerup's theory. If $M$ is of type $\mathrm{II}_1$ or is of type $\mathrm{III}_1$, then we need not do anything because Connes--Takesaki modules of automorphisms are always trivial. Hence we only need to consider the case when $M$ is of type $\mathrm{III}_0$ and the case when $M$ is of type $\mathrm{III}_\lambda$ ($0< \lambda <1$). Actually, as we will see in Remark \ref{lambda}, if $M$ is of type $\mathrm{III}_\lambda $ ($0<\lambda <1$), the Connes--Takesaki module of a flow cannot be faithful. Hence, the only problem is that how to handle the case when $M$ is of type $\mathrm{III}_0$. 

 Let $C$ be the center of $N$. First, we list up the form of the kernel of the action $\mathrm{mod}(\alpha ) \circ (\theta \mid_{C})$ of $\mathbf{R}^2$ on $C$. This is a closed subgroup of $\mathbf{R}^2$. Thus the kernel must be isomorphic to one of the following groups.
\[ 0,\ \mathbf{Z},\ \mathbf{Z}^2,\ \mathbf{R},\ \mathbf{R}\times \mathbf{Z},\ \mathbf{R}^2. \]
However, since $\theta \mid _{C}$ is faithful, the kernel cannot be isomorphic to $\mathbf{R}\times \mathbf{Z}$ or $\mathbf{R}^2$. We handle the other four cases separately.

\bigskip

We first consider the case when $\mathrm{ker}(\mathrm{mod}(\alpha ) \circ (\theta \mid _{C}))=0$. In this case, by an argument similar to that of the proof of Theorem 3.3 of Shimada \cite{S}, Lemma \ref{key} follows from a Rohlin type theorem due to Feldman \cite{F}. In the following, we will explain this theorem.

\bigskip

\textbf{Settings.} A subset $Q$ of $\mathbf{R}^d$ is said to be a cube if $Q$ is of the form
\[ [-s_1,t_1]\times \cdots \times [-s_d,t_d] \]
for some $s_1, \cdots , s_d, t_1,\cdots , t_d >0$. Let $Q$ be a cube of $\mathbf{R}^d$ and $T$ be a non-singular action of $\mathbf{R}^d$ on a Lebesgue space $(X, \mu )$. Then a measurable subset $F$ of $X$ is said to be a $Q$-set if $F$ satisfies the following two conditions.

\bigskip

(1) The map $Q\times F \ni (t,x) \mapsto T_t(x) \in X$ is injective.

(2) The set $T_QF:=\{ T_t(x) \mid t\in Q, x\in F \}$ is measurable and non-null.

\bigskip

In this setting, the following theorem holds.
\begin{theo} \textup{ (p.410 of Feldman \cite{F}, Thoerem 1 of Feldman--Lind \cite{FL})} 
\label{rp}
Let $T$ be a  free non-singular action of $\mathbf{R}^d $ on the standard probability space $(X, \mu)$. Then for any $\epsilon >0$ and for any cube $P$ of $\mathbf{R}^d$, there exists a large cube $Q$ and a $Q$-set $F$ of $X$ with 
\[ \mu (T_{\bigcap _{t\in P} (t+Q)}F) >1-\epsilon .\]
\end{theo}
The proof is written in Feldman \cite{F2}. However, his paper is privately circulated. Hence we explain the outline of the proof in Appendix (Section \ref{appendix}), which is based on Theorem 1 of Feldman--Lind \cite{FL} and Lind \cite{L}. As written in the proof of Theorem 1.1 (a) of Feldman \cite{F} (p.410 of Feldman \cite{F}), it is possible to introduce a measure $\nu $ on $F$ so that the map $Q\times F\ni (t,x) \mapsto T_t(x) \in T_QF$ is a non-singular isomorphism. The measure $\nu $ is defined in the following way. Set
\[ \mathcal{M}:= \{ A\subset F\mid T_Q(A) \ \mathrm{is} \ \mathrm{measurable} \ \mathrm{with} \ \mathrm{respect} \ \mathrm{to} \ \mu \}.\]
Then $\mathcal{M}$ is a $\sigma$-algebra of $F$ and it is possible to define a measure $\nu $ on $F$ by
\[ \nu (A):=\frac{\mu (T_QA)}{\mu(T_QF)}\]
for $A\in \mathcal{M}$. Then the map $(t,x) \mapsto T_t(x)$ is a non-singular isomorphism with respect to $\mathrm{Lebesgue}\otimes \nu$ and $\mu | _{T_QF}$. These things are written in p.410 of Feldman \cite{F} and the proof may be written in Feldman--Hahn--Moore \cite{FHM}. In this paper, for reader's convenience, we present a proof of what we will use (Propositions \ref{A.4} and \ref{A.5} of Appendix).
\begin{lemm}
When $\mathrm{ker}(\mathrm{mod}(\alpha )\circ (\theta | _C))$ is zero, Lemma \ref{key} holds.
\end{lemm}
\begin{proof}
Think of $C$ as $L^\infty (X, \mu ) $ for some probability measured space $(X, \mu )$. Let $T$ be an action of $\mathbf{R}^2$ defined by the following way.
\[ f\circ T_{(s,t)}=\theta _{-s} \circ \tilde{\alpha}_{-t} (f) \]
for $f\in L^\infty (X, \mu )$, $(s,t) \in \mathbf{R}^2$. Fix a natural number $n\in \mathbf{N}$. Set $P:=[-n,n]^2$. Then by Theorem \ref{rp}, there exists a large cube $Q$ and a $Q$-set $F$ of $X$ with 
\[ \mu (T_{\bigcap _{t\in P}(t+Q)}F)>1-\frac{1}{n}.\]
 Define a function $u_n $ on $X$ by the following way.
\[
  u_n = \begin{cases}
    e^{-ipt} & (x=T_{(s,t)}(y), \ (s,t)\in Q, \ y\in F) \\
    1 & (otherwise).
  \end{cases}
\]
Then by Proposition \ref{A.4}, the function $u_n$ is Borel measurable. Then for $x \in T_{\bigcap _{t\in P}(t+Q)}F$ and $(s,t)\in P$, we have 
\[ \theta _s(u_n)(x) =u_n(x),\]
\[ \tilde{\alpha }_t(u_n)(x)=e^{ipt}u_n(x).\]
Hence we have
\begin{align*}
\| \theta _s(u_n)-u_n \| _\mu^2 &\leq 4\mu (X\setminus T_{\bigcap _{t\in P}(t+Q)}F) \\
                                &\leq \frac{4}{n+1}
\end{align*}
for $s\in [-n,n]$. By the same computation, we have
\[\|  \tilde{\alpha }_t(u_n) -e^{ipt}u_n \| _\mu ^2\leq \frac{4}{n+1}\]
for $t\in [-n,n]$. Hence the sequence $\{ u_n \}$ of unitaries of $C$ satisfies the conditions in Lemma \ref{key}.
\end{proof}

\bigskip

Next, we consider the following case.

\begin{lemm}
\label{z2}
When $\mathrm{ker}(\mathrm{mod}(\alpha )\circ (\theta \mid _{C} ))$ is isomorphic to $\mathbf{Z}^2$, Lemma \ref{key} holds.
\end{lemm}
 In this case, there exist two pairs $(p_1, q_1 )$, $(p_2,q_2)$ of non-zero real numbers with $\mathrm{ker}(\mathrm{mod}\circ \theta )=\mathbf{Z}(p_1,q_1)\oplus \mathbf{Z}(p_2,q_2)$. Here, we use our assumption that $\mathrm{mod}(\alpha ) $ is faithful for showing $q_i\not=0$. Set $\sigma _t:=\theta _{q_1t}\circ \tilde {\alpha }_{p_1t}$. In order to show Lemma \ref{z2}, it is enough to show the following lemma.

\begin{lemm}
\label{10}
 For each $r\in \mathbf{R}$, there exists a sequence of unitaries $\{ u_n \}$ of $N$ which satisfies the following conditions.

\textup{(1)} We have $\| [u_n,\phi]\| \to 0$ for any $\phi \in N_*$.

\textup{(2)} We have $\theta _s(u_n)-u_n\to 0$ compact uniformly for $s\in \mathbf{R}$ in the strong* topology.

\textup{(3)} We have $\sigma _t(u_n)-e^{irt}u_n \to 0$ compact uniformly for $t\in \mathbf{R}$ in the strong* topology.

\end{lemm}

In order to show this lemma, we need to prepare some lemmas.

\begin{lemm}
 The action $\theta$ on $ C^\sigma $ is ergodic and has a period $p\in (0, \infty)$.
 \label{11}
\end{lemm}

\begin{proof}
 Ergodicity follows from the ergodicity of $\theta: \mathbf{R}\curvearrowright C$. We show that the restriction of $\theta$ on $C^\sigma $ has a period. We first note that a Borel measurable map $T$ from $\mathbf{T}$ to itself which commutes with every translations of the torus must be a translation because we have  $T(\gamma +t)-t=T(\gamma )$ for $t\in \mathbf{T}$ and for almost all $\gamma \in \mathbf{T}$. Now, we show that $C^\sigma \not=\mathbf{C}$. Assume that $C^\sigma $ were isomorphic to $\mathbf{C}$. Then since $\theta $ would commute with $\sigma$, which is a translation flow on the torus. Hence $\theta $ would be also a translation on the torus. Hence $\theta \circ \sigma $ would define a group homomorphism from $\mathbf{R}^2$ to the group of translations of the torus, which is isomorphic to $\mathbf{T}$. Hence the kernel of $\theta \circ \sigma $ would be isomorphic to $\mathbf{R}\times \mathbf{Z}$, which would contradict to the faithfulness of $\theta $. Combining this with the ergodicity of $\theta $, we have $\theta \mid _{C^\sigma}$ is non-trivial. Since $\mathrm{mod}(\alpha _{p_2})=\theta _{-q_2}\mid _{C}$, we have $(\sigma _{p_2/p_1}\circ \theta _{q_2-p_2q_1/p_1})\mid _C=\mathrm{id}_C$. Since $(p_1,q_1)$ and $(p_2,q_2)$ are independent, this $\theta \mid _{C^\sigma} $ has a non-trivial period.  
\end{proof}

By this lemma, we may assume the following. 

\bigskip

(1) We have  $C^\sigma =L^\infty (\mathbf{T}_p)$, where $\mathbf{T}_p$ is the torus of length $p$, which is isomorphic to $[0,p)$ as a measured space.

(2) We have $\theta _t(f) =f(\cdot -t)$ for $f\in L^\infty (\mathbf{T}_p)$, $t\in \mathbf{R}$.

\bigskip

Let 
\[ N=\int _{[0,p)}^\oplus N_\gamma \ d\gamma \]
be the direct integral decomposition of $N$. For $\gamma _1, \gamma _2 \in \mathbf{R}$,  $N_{\gamma_1}$ and $N_{\gamma _2}$  are mutually isomorphic by the following map.
\begin{equation*} 
\theta _{\gamma _2-\gamma _1, \gamma _1}: N_{\gamma _1}\to N_{\gamma _2},
\end{equation*}
\begin{equation*}
 \theta _{\gamma _2-\gamma _1, \gamma _1}(x_{\gamma _1})=(\theta _{\gamma _2-\gamma _1}(x))_{\gamma _2}
 \end{equation*}
for $x= \int _{[0,p) }^\oplus x_\gamma d\gamma \in N$. These $\theta_{\gamma_1, \gamma_2}$'s satisfy the following two conditions.

\bigskip

\textbf{Conditions.}

(1) The equality $\theta _{0, \gamma}=\mathrm{id}_{N_\gamma }$ holds for each $\gamma \in [0,p)$.

(2) The equality $\theta _{\gamma _3-\gamma _2, \gamma _2}\circ \theta_{\gamma _2-\gamma _1, \gamma _1}=\theta _{\gamma _3-\gamma_1,\gamma _1}$ holds for each $\gamma_1, \gamma _2, \gamma _3\in \mathbf{R}$.

\bigskip

By these $\theta _{\gamma _1, \gamma _2}$'s, all $N_\gamma$'s are mutually isomorphic. Thus it is possible to think of $N$ as $N_0\otimes L^\infty ([0,p))$. 

\bigskip

Now, we need to consider the measurability of $\theta _{t, \gamma}$. 

\bigskip

\textbf{Fact.}

If we identify $N$ with $N_0 \otimes L^\infty ([0,p))$, the map $[0, p)^2\ni(t, \gamma ) \mapsto \theta _{t, \gamma }\in \mathrm{Aut}(N_0)$ is Lebesgue measurable.  

\bigskip

Although this fact is probably well-known for specialists, for the reader's convenience, we present the proof in Appendix (Section \ref{appendix}). 

By measurability of $\theta _{t, \gamma }$, Lusin's theorem and Fubini's theorem, for almost all $\gamma \in [0,p)$, the map $t\mapsto \theta _{-t,t+\gamma}$ and $t\mapsto \theta _{t,\gamma}$ are also Lebesgue measurable. We may assume that $\gamma =0$ and we identify $N_{\gamma _1 }$ with $N_0$ by $\theta _{\gamma _1,0}$ for all $\gamma _1\in [0,p)$, that is, if we think of $N $ as the set of all essentially bounded weak * Borel measurable maps from $[0,p)$ to $N_0$, then the set of constant functions is the following set.
\[ \{ \int_{[0,p)}^\oplus \theta _{\gamma , 0}(x_0)\ d\gamma \mid x_0\in N_0 \} . \]
Take a normal faithful state $\phi _0$ of $N_0$. Then 
\[ \phi:=\frac{1}{p}\int _{[0,p)}^\oplus \phi _0 \circ \theta _{-\gamma , \gamma } \ d\gamma \]
is a normal faithful state on $N$. Choose $\phi _1, \cdots, \phi _n\in N_*$, $\epsilon >0$ and $T>0$. Then by the above identification of $N_*$ with $L^1_{(N_0)_*}([0, p))$, each $\phi _k$ is a Lebesgue measurable map from $[0, p)$ to $(N_0)_*$. Hence it is possible to approximate each $\phi _k$ by Borel simple step functions by the following way.
\[ \| \phi _k -\sum _{i=1}^{l_k}\phi _{k,i}\circ \theta _{-\gamma, \gamma}\chi _{I_i}(\gamma ) \| <\epsilon. \]
for each $k\in \{1, \cdots , n\}$. Here, $\phi _{k,i}\in{ N_0}_*$ for $i=1, \cdots , l_k$, $\{ I_i \} _{i=1}^{l_k}$ is a Borel partition of $[0, p)$.  Next, we look at actions on $N_0$. Let 
\[ \theta _p=\int _{[0,p)}^\oplus  \theta _p^\gamma \ d\gamma , \]
\[ \sigma _t =\int _{[0,p)}^\oplus \sigma _t^\gamma \ d\gamma , \]
\[ \tau =\int _{[0,p)}^\oplus \tau _\gamma \ d\gamma \]
be the direct integral decompositions. Since $\theta $ is trace-scaling and $\tilde{\alpha }$ is trace-preserving, $\sigma $ is trace-scaling. Hence for almost all $\gamma \in [0,p)$, $\sigma ^\gamma $ is $\tau _\gamma$ -scaling. Thus we may assume that $\sigma ^0$ is $\tau_0$-scaling. In order to show Lemma \ref{10}, it is enough to show the following lemma.

\begin{lemm} In the above context, for real number $r\in \mathbf{R}$, there exists a unitary $u_0$ of $N_0$ which satisfies the following conditions.
\label{13}

\textup{(1)} We have $\| [u_0, \phi _{k,i}] \| <\epsilon /(pl_k)$ for $k=1, \cdots , n$, $i=1, \cdots , l_k$.

\textup{(2)} We have $\| \theta _{mp}^0(u_0)-u_0\| _{\phi _0}^{\sharp }<\epsilon /p$ for $m\in \mathbf{Z}$, $| m| \leq p/T+2$.

\textup{(3)} We have $\| \sigma ^0_t(u_0)-e^{-irt}u_0 \| _{\phi _0}^{\sharp}<\epsilon /p$ for all $t\in [-T,T]$. 
\end{lemm} 

First, we show that Lemma \ref{13} implies Lemma  \ref{10}.

\bigskip

\textit{Proof of Lemma \ref{13} $\Rightarrow$ Lemma \ref{10}.}
 Assume that there exists a unitary $u_0$ in $N_0$ which satisfies the conditions in Lemma \ref{13}. We set
\[ u_\gamma :=\theta _{\gamma ,0}(u_0), \] 
\[ u:=\int _{ [0,p) }^\oplus u_\gamma \ d\gamma .\]
Fix $t\in [-T, T]$ and $\gamma \in [0,p)$. For each $\gamma \in [0,p)$, choose $m_\gamma \in \mathbf{Z}$ so that $-t+m_\gamma p+\gamma \in [0,p)$. Then we have
\begin{align*}
(\theta _t(u))_\gamma &=\theta _{t,-t+\gamma}(u_{-t+\gamma}) \\
                      &=\theta _{t, -t+\gamma}(u_{-t+\gamma +m_\gamma p}) \\
                      &=\theta _{\gamma ,0}\circ \theta ^0_{m_\gamma p}\circ \theta _{t-m_\gamma p -\gamma ,-t+m_\gamma p+\gamma}(u_{-t+m_\gamma p+\gamma} )\\
                      &=\theta _{\gamma ,0}\circ \theta ^0_{m_\gamma p}(u_0).
\end{align*}
Hence we have 
\begin{align*}
 \| \theta _t(u)-u\| _{\phi }^\sharp &= \int _{[0,p)}  \| (\theta _t(u))_\gamma -u_\gamma \| _{\phi _0 \circ \theta _{-\gamma ,\gamma }} ^\sharp \ d\gamma \\
                                                   &=\int _{[0,p) } \| \theta _{\gamma ,0}\circ \theta ^0_{m_\gamma p}(u_0 ) -\theta _{\gamma, 0} (u_0) \| _{\phi_0 \circ \theta _{-\gamma ,\gamma }} ^\sharp \ d\gamma \\
                                                   &=\int _{[0,p)} \| \theta ^0_{m_\gamma p }(u_0)-u_0 \| _{\phi _0}^\sharp \ d\gamma \\
                                                   &<\int _{[0,p)} \frac{\epsilon}{p} \ d\gamma \\
                                                   &=\epsilon .
\end{align*}
Here we use that $| m_\gamma | \leq T/p +2$ in the fourth inequality of the above estimation. By the same argument, we also have
\[ \| \sigma _t(u)-e^{-irt}u \| _\phi ^\sharp <\epsilon \]
for $t\in [-T,T]$.
We also have
\begin{align*}
\| [u, \phi _k ]\| &\leq 2\| \phi _k -\sum _{i=1}^{l_k}\phi _{k,i} \otimes \chi _{I_i} \| +\|[u, \sum _{i=1}^{l_k}\phi _{k,i} \otimes \chi _{I_i} ]\| \\
                                 &<2\epsilon +\sum _{i=1}^{l_k}\| [u, \phi _{k,i} \otimes \mathrm{id} ]\| \\
                                 &=2\epsilon +\sum _{i=1}^{l_k}\int _{[0,p)} \| [\theta _{\gamma, 0}(u_0), \phi _{k,i}\circ \theta _{-\gamma, \gamma} ]\| \ d\gamma \\
                                 &=2\epsilon +\sum _{i=1}^{l_k}\int _{[0,p)} \| [u_0, \phi _{k,i}] \| \ d\gamma \\
                                 &<2\epsilon +\sum _{i=1}^{l_k}\int _{[0,p)} \frac{\epsilon }{pl_k} \ d\gamma \\
                                 &=3\epsilon .
\end{align*}
Thus Lemma \ref{10} holds.
\qed

\bigskip

In order to prove Lemma \ref{13}, we first rewrite the lemma in a simpler form. To do this, we show that there exists a number $s\in (0,1)$ with $(\theta ^0_p \circ \sigma ^0_s) \mid _{Z(N_0)}=\mathrm{id}$. Since the restriction of $\sigma^0$ on the center of $N_0$ has a period $1$ and is ergodic, we may assume that $Z(N_0)$ is isomorphic to $L^\infty ([0,1))$, which is canonically identified with $L^\infty(\mathbf{T})$, $\sigma ^0_s(f)=f(\cdot -s)$ for $s\in \mathbf{R}$, $f\in L^\infty (\mathbf{T})$. By this identification, $\theta ^0_p$ commutes with all $\sigma ^0_s$'s. Hence $\theta ^0$ is a translation on the torus. Thus there exists a unique $s\in (0,1)$ with $(\theta^0_p \circ \sigma ^0 _s) \mid _{Z(N_0)}=\mathrm{id}$. Set $\beta ^0:=\theta _p ^0 \circ \sigma ^0_s$. The proof of Lemma \ref{13} reduces to that of the following lemma.
\bigskip

\begin{lemm}
\label{14}
The action $\{ \beta ^0_m \circ \sigma ^0_t\}_{(m,t)}$ of $\mathbf{Z}\times \mathbf{R}$ on $N_0$ has the Rohlin property.
\end{lemm}
\textit{Proof of Lemma \ref{14}$\Rightarrow$ Lemma \ref{13}.} Assume that the action $\{ \beta^0_m\circ \sigma^0_t\}_{m,t}$ has the Rohlin property. Then there exists a unitary element $u_0$ of $N_0$ with the following conditions.

\bigskip

(1) We have $\|[u_0, \phi_{k,i}]\| <\epsilon/(pl_k)$ for $k=1, \cdots ,n$, $i=1, \cdots ,l_k$.

(2) We have $\| \beta ^0_{m} (u_0)-e^{-irms}u_0\| _{\phi _0}^{\sharp} <\epsilon /(2p)$ for $m\in \mathbf{Z}$, $|m|\leq p/T+2$.

(3) We have $\| \sigma ^0_t(u_0)-e^{-irt}u_0\| _{\phi _0\circ \theta _{mp}^0}^\sharp <\epsilon /(2p) $ for $t\in [-(1+s)(T+2p),(1+s)(T+2p)]$, $m\in \mathbf{Z}$, $|m|\leq p/T+2$.

\bigskip

Since $\beta^0_{m}=\theta ^0_{mp}\circ \sigma ^0_{ms}$, we have 
\begin{align*}
\| \theta ^0_{mp}(u_0)-u_0 \| &=\|e^{-irms}\theta ^0_{mp}(u_0)-e^{-irms}u_0 \|_{\phi _0}^\sharp \\
                                       &\leq \| \theta _{mp}^0 (e^{-irms}u_0-\sigma _{ms}^0(u_0))\|_{\phi_0}^\sharp +\| \beta _{m}^0(u_0)-e^{-irms}u_0\|^\sharp _{\phi _0} \\
                                      &=\| e^{-irms}u_0-\sigma _{ms}^0(u_0)\|_{\phi _0\circ \theta _{mp}^0}^\sharp + \| \beta _{m}^0(u_0)-e^{-irms}u_0\|^\sharp _{\phi _0}\\
                                      &<\frac{\epsilon}{2p}+\frac{\epsilon}{2p} \\
                                      &=\frac{\epsilon}{p}
\end{align*}
for $m\in \mathbf{Z}$, $|m|\leq p/T+2$. Thus Lemma \ref{13} holds.
\qed

\bigskip

In order to show Lemma \ref{14}, we need further to reduce the lemma to a simpler statement. Let 
\[ N_0=\int _{[0,1)}^\oplus (N_0)_\gamma \ d\gamma \]
be the direct integral decomposition of $N_0$ over the center of $N_0$. For each $\gamma _1, \gamma _2\in [0,1)$, there exists an isomorphism from $(N_0)_{\gamma _1}$ to $(N_0)_{\gamma_2}$ defined by 
\[ \sigma ^0 _{\gamma _2-\gamma _1, \gamma _1}((x_0)_{\gamma _1})=(\sigma ^0 _{\gamma _2-\gamma _1}(x_0))_{\gamma _2} \]
for $x_0=\int _{[0,1)}^\oplus (x_0)_\gamma d\gamma \in N_0$. These $\sigma ^0 _{\gamma _2-\gamma _1, \gamma _1}$'s satisfy the conditions similar to conditions (1) and (2) of $\theta _{t, \gamma}$ (See Conditions between Lemma \ref{11} and Lemma \ref{13}).

We identify $(N_0)_\gamma$'s with $(N_0)_0$ by $\sigma ^0_{\gamma, 0}$. Choose a normal faithful state $\psi_0 $ of $(N_0)_0$. Set
\[ \psi:=\int _{[0,1)} ^\oplus \psi _0\circ \sigma ^0_{-\gamma, \gamma} \ d\gamma . \]
This is a normal faithful state on $N_0$. Choose $\psi _1, \cdots \psi _n \in (N_0)_*$, $\epsilon >0$ and $T>0$. By the same argument as above, we may assume that $\psi _k$'s are simple step Borel functions.
\[ \psi _k =\sum _{i=1}^{l_k} \psi _{k,i} \circ \sigma ^0_{-\gamma , \gamma }\chi _{I_i}(\gamma ) \]
for $k=1, \cdots , n$. Here, $\psi _{k,i}\in (N_0)_*$, $\{ I_i\}_{i=1}^{l_k}$ are partitions of $[0,1)$. Since $\beta ^0$ and $\sigma ^0$ fix the center of $N_0$, they are decomposed into the following form.
\[ \beta ^0 =\int _{[0,1)}^\oplus \beta^{0,\gamma } \ d\gamma ,\]
\[ \sigma ^0_1=\int _{[0,1)}^\oplus \sigma ^{0,\gamma } \ d\gamma .\]
Then for each $\gamma \in [0,1)$, $ \{ \beta ^{0,\gamma }_n \circ \sigma ^{0,\gamma }_m \}_{(n,m)\in \mathbf{Z}^2} $
defines an action of $\mathbf{Z}^2$ on $(N_0)_\gamma$, which is isomorphic to the AFD factor of type $\mathrm{II}_\infty $. We show the following lemma, which is essentially important, that is, assumption that $\mathrm{mod}(\alpha )$ is faithful is essentially used for showing this lemma.

\begin{lemm}
\label{15}
 For almost all $\gamma \in [0,1)$, the action  $\{ \beta ^{0,\gamma }_n \circ \sigma ^{0,\gamma }_m \}$ is trace-scaling for $(n,m)\not=0$.
\end{lemm}

\begin{proof} Take a pair $(n,m)\not=0$. By definition of $\beta^0$ and $\sigma ^0$, we have
\begin{align*}
 \beta ^0_n \circ \sigma ^0_m &=(\theta _{np} \circ \sigma _{ns})^0\circ \sigma ^0_m \\
                              &=(\theta _{np} \circ \sigma _{ns+m})^0 \\
                              &=(\theta _{np} \circ \theta _{(ns+m)q_1} \circ \tilde{\alpha} _{(ns+m)p_1})^0 \\
                              &=(\theta _{np + (ns+m)q_1} \circ \tilde{\alpha} _{(ns+m)p_1})^0. 
\end{align*}
If $n=0$, we need not show anything . Assume that $n\not=0$. Then since $\theta _{np}$ is not identity on the center of $N_0$, $\sigma _{ns+m}$ is not identity on $Z(N_0)$ by looking at the first equation. Hence $(ns+m)p_1\not=0$.  Thus, by the faithfulness of $\mathrm{mod}(\alpha)$ and the last equation, we have $np +(ns+m)q_1\not =0$. Hence $\theta  _{np +(ns+m)q_1}$ scales $\tau $. Besides, $\tilde{\alpha }$ preserves $\tau$. Hence we may assume that $\beta ^0 _n\circ \sigma ^0_m$ scales $\tau _0$. Hence if we decompose $\tau ^0$ by
\[ \tau^0 =\int _{[0,1)}^\oplus \tau ^{0,\gamma } \ d\gamma , \]
$(\beta ^0 _n \circ \sigma ^0_m)^\gamma $ scales $\tau ^{0,\gamma}$ for almost all $\gamma \in [0,1)$. 
\end{proof}

From Lemma \ref{15}, we may assume that $\{ \beta ^{0,0}_n \circ \sigma ^{0,0}_m \}$ is trace-scaling. Now, we will return to prove Lemma \ref{14}, which completes the proof of Lemma \ref{10}.

\bigskip

\textit{Proof of Lemma \ref{14}.}
Let $\psi _0$, $\epsilon$ and $T$ be as explained after the statement of Lemma \ref{14}. By Lemma \ref{15}, the action $\{ \beta ^{0,0}_n \circ \sigma ^{0,0} _m\}$ is centrally outer, and hence has the Rohlin property. Hence, for $A,B\in \mathbf{N}$ with $4(T+1)^2/\epsilon ^2 <B$ and $A>1/\epsilon ^2 $, there exists a family of projections $\{ e_{n,m}\}_{n=1,\cdots, B}^{m=1, \cdots A}$ of $N_{0,0}$ which satisfies the following conditions.

\bigskip

(1) The projections are mutually orthogonal.

(2) We have 
\begin{gather*}
\sum _{n,m} e_{n,m}=1,  \\
\| \sum _{1\leq n\leq B, m=1,A} e_{n,m} \| _{\psi _0 +\psi _0 \circ \beta ^{0,0} }^\sharp \leq 2/\sqrt{A}, \\
\| \sum _{1\leq m\leq A,n\geq B-(T+1),n\leq T+1} e_{n,m}\| _{\psi _0+\sum _{l=-[T]-1}^{[T]+1} \psi _0\circ \sigma ^{0,0}_l }^\sharp \leq 2(T+1)/\sqrt{B}. 
\end{gather*}
Here, $[T]$ is the maximal natural number which is not larger than $T$.

(3) We have $\| [e_{n,m}, \psi _{k,i}]\|<\epsilon/(ABl_k)$ for $n=1, \cdots B$, $m=1, \cdots , A$, $i=1, \cdots l_k$, $k=1, \cdots n$.

(4) We have $\| \sigma ^{0,0}_l(e_{n,m})-e_{n+1,m} \| _{\psi _0}^\sharp <\epsilon /(AB)$ for $n,m$, $l \in \mathbf{Z}$ with $| l| \leq T+1$, $n\leq B-(T+1)$.

(5) We have $\| \beta ^{0,0}(e_{n,m})-e_{n,m+1}\| _{\psi _0}^\sharp <\epsilon /(AB)$ for $n,m\in \mathbf{Z}$ with $m\not =A$. 

\bigskip

Here, we define $e_{B+1,m}=e_{1,m}$ for $m=1, \cdots , A$, $e_{n,A+1}=e_{n,1}$ for $n=1, \cdots , B$. For $(s,t) \in \mathbf{T}\times \mathbf{R}$, we set
\[ u_\gamma :=e^{2\pi it\gamma }\sum _{n,m}e^{2\pi i (nt+ms)}\sigma ^0_{\gamma, 0 }(e_{n,m}) \]
for $\gamma \in [0,1)$. We also set
\[ u:=\int _{[0,1)}^\oplus u_\gamma \ d\gamma \in U(N_0).\]
The above conditons (2) and (4) ensure that we can almost control $\sigma^{0,0}$, which is useful to show that $\sigma ^0_q(u)$ is close to $e^{-2\pi itq}u$. Conditions (2) and (5) is useful to show that $\beta ^0(u)$ is close to $e^{-2\pi is}u$. Condition (3) is useful to show that $[u, \psi _k]$ is small.

\bigskip

By condition (3) of the above, we have
\begin{align*}
\| [u, \psi _k ] \| &\leq \int _{[0,1)} \sum _{n,m}\sum _{i=1}^{l_k} \| [\sigma _{\gamma ,0}^0(e_{n,m}), \psi _{k,i} \circ \sigma _{-\gamma,\gamma}^0] \| \ d\gamma \\
                                  &=\int _{[0,1)} \sum _{n,m}\sum _{i} \| [e_{n,m}, \psi _{k,i}] \| \ d\gamma \\
                                  &<\int _{[0,1)} \sum _{n,m}\sum _{i} \frac{\epsilon}{ABl_k} \ d\gamma \\
                                  &=\epsilon .
\end{align*}
By conditions (2) and (5), we have
\begin{align*}
&\| \beta ^0 (u)-e^{-2\pi is}u \| _\psi ^\sharp \\
&=\int _{[0,1)} \| \beta ^{0,\gamma } \circ \sigma ^0_{\gamma, 0}( \sum _{m,n} (e^{2\pi i ((n+\gamma )t+ms)}e_{n,m})) \\
& -\sum _{m,n} e^{2\pi i ((n+\gamma )t +(m-1)s)}\sigma ^0_{\gamma , 0}(e_{n,m}) \| _{\psi _0 \circ \sigma ^0_{-\gamma, \gamma}}^\sharp \ d\gamma \\
&=\int _{[0,1)} \| \sigma ^0 _{\gamma, 0}(\beta ^{0,0}(\sum _{m,n}e^{2\pi i((n+\gamma )t+ms)}e_{n,m}) \\
&-\sum _{m,n} e^{2\pi i ((n+\gamma )t+(m-1)s)}e_{n,m}) \| _{\psi _0\circ \sigma ^0_{-\gamma , \gamma}}^\sharp \ d\gamma \\
&\leq \sum _{m,n, m\not =A}\int _{[0,1)} \| e^{2\pi i ((n+\gamma )t+ms)}\sigma ^0_{\gamma ,0}(\beta ^{0,0}(e_{n,m})-e_{n,m+1}) \| _{\psi _0 \circ \sigma ^0_{-\gamma, \gamma}}^\sharp  \ d\gamma \\
&+ \int _{[0,1)} \| \sigma ^0_{\gamma, 0} (\beta ^{0,0}(\sum _n e_{n,A})-\sum _n e_{n,A+1})\| _{\psi _0\circ \sigma ^0 _{-\gamma , \gamma}}^\sharp \ d\gamma \\
&\leq \sum _{m,n,m\not =A} \int _{[0,1)} \| \beta ^{0,0}(e_{n,m})-e_{n,m+1} \| _{\psi _0}^\sharp \ d\gamma +2/\sqrt{A} \\
&<AB(\frac{\epsilon}{AB})+2\epsilon \\
&=3\epsilon .
\end{align*}
Condition (2) is used in the fourth inequality and condition (5) is used in the fifth inequality. Next, we will compute $\| \sigma ^0 _q (u)-e^{2\pi itq}u \| _\psi ^\sharp$ for $q\in [-T,T]$. In order to do this, the following observation is useful. Let $\gamma \in [0,1)$ and let $q\in [-T,T]$. Choose $l_\gamma \in \mathbf{Z}$ so that $\gamma -q +l_\gamma \in [0,1)$. Then we have
\begin {align*}
(\sigma ^0_q(u))_\gamma &=\sigma ^0_{q, \gamma -q}(u_{\gamma -q+l_\gamma}) \\
                        &=\sigma ^0_{\gamma,0} \circ \sigma ^{0,0}_{l_\gamma} \circ \sigma ^0_{q-l_\gamma -\gamma, \gamma-q+l_\gamma}(u_{\gamma -q+l_\gamma}) \\
                        &=\sigma ^0_{\gamma , 0} \circ \sigma ^{0,0}_{l_\gamma} (e^{2\pi it(\gamma -q+l_\gamma )}\sum _{n,m}e^{2\pi i(nt+ms)}e_{n,m}) \\
                        &=e^{2\pi it(\gamma -q +l_\gamma )}\sigma ^0_{\gamma, 0} \circ \sigma ^{0,0}_{l_\gamma }(u_0).
\end{align*}
By conditions (2) and (4), we have
\begin{align*}
&\| \sigma ^0_q(u)-e^{-2\pi itq}u\| _{\psi}^\sharp \\
&= \int _{[0,1)} \| \sigma ^0 _{\gamma , 0} (e^{2\pi it(\gamma -q+l_\gamma )} \sigma ^{0,0}_{l_\gamma }(u_0) -e^{-2\pi itq}u_\gamma) \| _{\psi _0 \circ \sigma _{-\gamma, \gamma}^0}^\sharp \ d\gamma \\
&= \int _{[0,1)} \| \sum _{n,m}(e^{2\pi i((\gamma -q+l_\gamma +n)t+ms)}\sigma ^{0,0} _{l_\gamma }(e_{n,m})-e^{2\pi i((\gamma +n-q)t+ms)}e_{n,m}) \| _{\psi _0}^\sharp  \ d\gamma \\
&\leq \int _{[0,1)} \sum _{m,l_\gamma<n\leq B-1} \| e^{2\pi i((\gamma -q+n)t+ms)}\sigma ^{0,0} _{l_\gamma }(e_{n-l_\gamma ,m}) \\
&-e^{2\pi i((\gamma +n-q)t+ms)}e_{n,m}) \| _{\psi _0}^\sharp \ d\gamma +2\epsilon \\
&<(\frac{\epsilon}{AB})AB +2\epsilon \\
&=3\epsilon .
\end{align*}
Condition (2) is used in the third inequality and condition (4) is used in the fourth inequality. Thus $\{ \sigma ^0_q \circ \beta ^0_m\}_{(q,m)}$ has the Rohlin property. Thus Lemma \ref{14} holds.
\qed

\bigskip

\begin{lemm}
\label{r}
When $\mathrm{ker}(\mathrm{mod}(\alpha _s) \circ (\theta \mid_C)_t)\cong \mathbf{R}$, Lemma \ref{key} holds. 
\end{lemm}
\begin{proof}
There exists $(p,q)\in (\mathbf{R}\setminus \{0\})^2$ with $\mathrm{ker}(\mathrm{mod}(\alpha ) \circ \theta \mid _C)=(p,q)\mathbf{R}$. Set $\sigma _t:=\theta _{qt}\circ \tilde{\alpha} _{pt}$ for $t\in \mathbf{R}$. In order to show our lemma, it is enough to show that for each $r\in \mathbf{R}$, the action $\sigma $ admits a sequence of unitaries which satisfies the same conditions as in Lemma \ref{10}.
 Take a normal faithful state $\phi _0$ of $N$, $\phi _1, \cdots, \phi _n\in N_*$ with $\|\phi _k \| =1$ ($k=1, \cdots,n$), $\epsilon >0$ and $T>0$. Think of $C$ as a standard probability measured space $L^\infty (\Gamma, \mu )$.  Let 
\[ N=\int _{\Gamma}^\oplus N_{\gamma} \ d\mu (\gamma ), \]
\[ \phi _k =\int _\Gamma^\oplus \phi _{k,\gamma} \ d\mu (\gamma) \]
($k=1, \cdots, n$) be the direct integral decompositions.
Then by Theorem \ref{rp} and Proposition \ref{A.5}, there exists a Borel subset $A$ of $\Gamma$ which satisfies the following three conditions.

\bigskip

(1) There exists a large cube $Q:=[-T',T']$ and a $Q$-set $Y$ such that $A=T_QY$ and the map $Q\times Y \ni (t,x) \mapsto T_t(x)\in A$ is injective.

(2) We have 
\[\mu (\bigcap _{t\in [-T,T]}T_tA) >1-\epsilon\]
and
\[\int  _{\bigcap _{t\in [-T,T]}T_tA}\| \phi _{k,\gamma}\| \ d\mu (\gamma )>1-\epsilon\]
for $k=1, \cdots ,n$.

(3) There is a measure $\nu $ on $Y$ such that the map $Q\times Y\ni (t,x) \mapsto T_t(x) \in A$ is a non-singular isomorphism with respect to $\mathrm{Lebesgue} \otimes \nu $ and $\mu$ (Note that two measures $\mu +\sum _k \int _{\Gamma}\| \phi _{k,\gamma}\| \ d\mu (\gamma )$ and $\mu $ are mutually equivalent).

\bigskip

Here, we do not assume the existence of invariant probability measures for $\theta \mid _C$. Then $N$ is isomorphic to 
\[ N_{\Gamma \setminus A}\oplus  \int _{ [-T',T']} ^\oplus N_s \ ds.\]
Here, 
\[ N_s=\int _{Y}^\oplus N_{(y,s)} \ d\nu (y) .\]
For $s,t\in [-T',T']$, $\theta $ defines an isomorphism $\theta _{s-t,t}$ from $N_t$ to $N_s$ by $\theta _{s-t,t}(x_t)=(\theta _{s-t}(x))_s$. As in Lemma \ref{10}, we identify $N_t$ with $N_0$ by this isomorphism. By this identification, we approximate $\phi _k$'s by simple step functions.
\[ \| \phi _k \chi _{A}-\sum _{i=1}^{l_k}\phi _{k,i}\circ \theta _{-t,t} \chi _{I_i}(t) \| <\epsilon \]
for $k=1,\cdots, n$, where $\phi _{k,i}\in (N_0)_*$ and $\{I_i\}_{i=1}^{l_k}$ are partitions of $[-T',T']$. Here, we note that it is possible to choose $\phi _{0,i}$'s so that they are positive.  This is shown by the following way. Since $\phi _0: [-T',T'] \to (N_0)_*$ is measurable, by Lusin's theorem, it is possible to choose a sufficiently large compact subset $K$ of $[-T',T']$ on which $\phi _0$ is continuous. Choose a finite partition $\{s_i\}_{i=1}^{l_0}$ of  $K$ so that for every $ s\in K$, there exists a number $i$ such that $\phi _0(s) $ is close to $\phi (s_i)$. It is possible to choose a partition $\{ I_i\}$ of $K$ so that $\phi _0(s)$ is close to $\phi _0(s_i)$ on $I_i$. Then $\sum \phi _{0}(s_i)\chi _{I_i}$ well approximates $\phi _0$. Since $\sigma $ fixes the center of $N$, this is decomposed into the direct integral.
\[ \sigma =\sigma _{\Gamma \setminus A} \oplus \int _{[-T',T']}\sigma ^{t} \ dt.\]
Since $\sigma $ scales the canonical trace on $N$, for almost all $t\in \mathbf{R}$, the action $\sigma ^t$ is trace-scaling, and hence has the Rohlin property by Theorem 6.18 of Masuda--Tomatsu \cite{MT}. Hence, by the same argument as in the proof of Lemma \ref{14}, it is possible to choose a unitary element $u_0$ of $N_0$ satisfying the following conditions.

\bigskip

(1) We have $\| [u_0, \phi _{k,i}]\| < \epsilon /(2l_kT')$  for $k=1, \cdots, n$, $i=1, \cdots, l_k$.

(2) We have $\| \sigma ^{0}_t(u_0)-e^{-ipt}u_0 \| _{\phi _{0,i}}^\sharp <\epsilon /(2l_0T')$ for $t\in[-T,T]$, $i=1, \cdots, l_0$.

\bigskip

Set $u_t:=\theta _{t,0}(u_0)$ for $t\in [-T',T']$ and set
\[ u:=\chi _{X\setminus A} \oplus \int _{[-T',T']}^\oplus u_t \ dt.\]
Hence by the same aregument as in the proof of Lemma \ref{13} $\Rightarrow $ Lemma \ref{10}, Lemma \ref{key} holds.
\end{proof}

\begin{lemm}
\label{z}
When  $\mathrm{ker}(\mathrm{mod}(\alpha _s)\circ (\theta \mid _C) _t) \cong \mathbf{Z}$, Lemma \ref{key} holds.
\end{lemm}
\begin{proof}
 Let  $(p,q )\in (\mathbf{R}\setminus \{0\})^2$ be a generator of $\mathrm{ker}(\mathrm{mod}(\alpha _s) \circ (\theta \mid_C)_t)$. Set $\sigma _t:=\theta _{qt}\circ \tilde{\alpha }_{pt}$ for $t\in \mathbf{R}$. Think of $C^\sigma $ as a standard probability space $L^\infty (\Gamma, \mu )$. We first show the following claim.
 
 \bigskip

\textbf{Claim}. The action $\theta :\mathbf{R}\curvearrowright C^\sigma $ is faithful (and hence is free).

\textit{Proof of Claim}. Assume that $\theta _t\mid _{C^\sigma}=\mathrm{id}_{C^\sigma}$. Then $\theta$ is decomposed into the direct integral over $\Gamma$. 
\[ \theta _t=\int _{ \Gamma}^\oplus \theta _t^\gamma \ d\mu (\gamma ) ,\]
\[ N=\int _{ \Gamma}^\oplus N_{ \gamma } \ d\mu (\gamma ). \]
We also decompose $\sigma $ by
\[ \sigma _s =\int _{\Gamma}^\oplus \sigma ^\gamma _s \ d\mu (\gamma ) .\]
Then for almost all $\gamma \in \Gamma$, $\{ \sigma ^\gamma _t\}_{t\in \mathbf{R}}$ defines a periodic ergodic action on the center of $N_\gamma$. Since the restriction of $\theta ^\gamma _t$ on the center of $N_\gamma $ commutes with that of $\sigma ^\gamma _s$'s, $\theta ^\gamma _t \mid _{Z(N_\gamma )}$ is of the form $\sigma ^\gamma _{s_\gamma}\mid _{Z(N_\gamma )}$. We show that there exists $s\in [0,1)$ such that $s_\gamma =s$ for almost all $\gamma $. Since we want to show the faithfulness of the action $\theta $,  we may  assume that $t\not =0$. We think of $C$ as a probability measured space $L^\infty (X, \mu_X)$. Then there exists a projection $p$ from $X$ to $\Gamma$ induced by the inclusion $L^\infty (\Gamma ) \to L^\infty (X)$. Let $T,S $ be two flows on $(X, \mu _X)$ defined by
$f(T_tx)=\theta _{-t}(f)(x)$, $f(S_sx)=\sigma _{-s}(f)(x)$ for $x\in X$, $f\in L^\infty(X, \mu _X)$. We may assume that $X$ is a separable compact Hausdorff space and $T$ and $S$ are continuous. We show that the set
\[ A_s= \{ x\in X \mid T _t\circ S _{r}(x)=x \ \mathrm{for} \  \mathrm{some}\ 0\leq r\leq s\} \]
is Borel measurable. Let $f:\mathbf{R}\times X \to X^2$ is a map defined by $f:\mathbf{R}\times X\ni (s,x) \mapsto (T _t\circ S _s(x),x)\in X^2$. Then we have  $A_s=\pi _X(f^{-1}(\Delta )\cap ([0,s] \times X))$, which is Borel measurable. Here, $\Delta $ is the diagonal set of $X\times X$ and $\pi _X :\mathbf{R} \times X \to X$ is the projection. 

Next we show that there exists $s \in [0,1)$ such that 
\[ B_s:=\{ x\in X \mid T_t\circ S_s(x)=x\} \]
 has a positive measure. If not, the map $s\to \mu _X(A_s)$ would be continuous. By the first part of this proof, for each $\gamma \in \Gamma$, if $x\in X$ satisfies $p(x)=\gamma$, then we have $x\in A_{s_\gamma}$. Hence $\bigcup _{s>0}A_s$ is full measure. On the other hand, since $t\not =0$, we have $\mu (A_0)=0$. Thus there would exist $s \in [0,1)$ with $\mu _X(A_s)=1/2$. However, this would contradict to the ergodicity of $\theta$. Thus there exists $s\in [0,1)$ with $\mu _X(B_s)>0$. 

By using the ergodicity of $\theta $ again, there exists $s\in [0,1)$ such that $B_s$ is full measure. 

Hence there exists $s\in [0,1)$ such that $\sigma _s\mid _C =\theta _t\mid _C$. Since $\mathrm{ker}(\mathrm{mod}(\alpha )\circ (\theta \mid _C))=(p,q)\mathbf{Z}$, we have $s=t=0$, which is a contradiction. Hence Claim is shown.
 \qed
 
 \bigskip

Now, we return to the proof of Lemma \ref{z}. For almost all $\gamma \in \Gamma$, the action $\sigma ^\gamma \mid _{Z(N_0)}$ is ergodic and has a period $1$, and $\sigma ^\gamma $ is trace-scaling. Hence this is the dual action of a modular automorphism of an AFD $\mathrm{III}_\lambda$ ($0<\lambda <1$) factor. Hence $\sigma ^\gamma $ has the Rohlin property. Hence by the same argument as in Lemma \ref{r}, our lemma is shown.
\end{proof}

\begin{rema}
\label{lambda}
 When $M$ is of type $\mathrm{III}_\lambda$, $0<\lambda <1$, then Connes--Takesaki module of a flow on $M$ cannot be faithful. This is shown by the following way. Since $\mathrm{mod}(\alpha )$ commutes with $\theta$, as we have seen, this is a homomorphism from $\mathbf{R}$ to $\mathbf{T}$. Hence $\mathrm{mod}(\alpha ) $ cannot be faithful.
 \end{rema}

\section{Remarks and Examples}
In this section, we present examples which have interesting properties.
\subsection{Model Actions}
\label{model}
In this subsection, we will construct model actions. If there were no flows with faithful Connes--Takesaki modules, then our main theorem would have no value. Hence, it is important to construct a flow which has  a given flow as its Connes--Takesaki module. 

\begin{prop}
Let $M$ be an AFD factor with its flow space $\{C, \theta \}$ and let $\sigma $ be a flow on $C$ which commutes with $\theta $ \textup{(}Here, we do not assume the faithfulness of $\sigma $\textup{)}. Then there exists a Rohlin flow $\alpha $ on $M$ with $\mathrm{mod}(\alpha )=\sigma $. 
\end{prop}
\begin{proof}
The proof is modeled after Masuda \cite{M}. 
 As in Corollary 1.3 of Sutherland--Takesaki \cite{ST3}, there exists an exact sequence 
\[ 1 \to\overline{ \mathrm{Int}}(M) \to \mathrm{Aut}(M) \to \mathrm{Aut}_\theta (C) \to 1,\]
and there exists a right inverse $s: \mathrm{Aut}_\theta (C) \to \mathrm{Aut}(M)$. The maps $p: \mathrm{Aut}(M)\to \mathrm{Aut}_\theta (C)$ and $s: \mathrm{Aut}_\theta (C) \to \mathrm{Aut}(M)$  are continuous. 

Hence for a flow $\sigma $ on $C$ commuting with $\theta $, the homomorphism $\alpha :=s\circ \sigma :\mathbf{R}\to \mathrm{Aut}(M)$ gives an action with its Connes--Takesaki module $\sigma $. If $\sigma $ is faithful, by our main theorem, this has the Rohlin property. Assume that $\sigma $ is not faithful.  Then $\mathrm{mod}(\alpha \otimes \beta )=\mathrm{mod}(\alpha )$ for a Rohlin flow $\beta $ on the AFD factor of type $\mathrm{II}_1$. Hence this $\alpha \otimes \beta$ does the job.
\end{proof}

For actions on the AFD  factor of type $\mathrm{II}_1$, strong cocycle conjugacy is equivalent to cocycle conjugacy because every automorphism of the AFD factor of type $\mathrm{II}_1$ is approximated by its inner automorphisms. 
However, for flows on some AFD factor of type $\mathrm{III}_0$, cocycle conjugacy does not always imply strong cocycle conjugacy. 

\begin{exam}
Let $(X, \mu )$ be a probability measured space defined by
\[ (X, \mu ):=(\prod _{m\in \mathbf{Z}}(\prod _{n\in \mathbf{Z}}(\{ 0,1\},\{ \frac{1}{2}, \frac{1}{2} \} ))).\]
Let $S$, $T$ be two automorphisms of $X$ defined by the following way.
\[ S(m\mapsto (n\mapsto {x_n}^m \in \{0,1 \} )) =(m\mapsto (n\mapsto {x_{n+1}}^m)),\]
\[ T(m\mapsto (n\mapsto {x_n}^m )) =(m\mapsto (n\mapsto {x_n}^{m+1})).\]
Then both $S$ and $T$ are ergodic and satisfy $S \circ T= T\circ S$. Let $\beta _1$, $\beta_2$ be two flows on $L^\infty(\mathbf{T})$ satisfying the following conditions. 

\bigskip

\textup{(1)} Two flows are faithful.

\textup{(2)} The flow $\beta _1 $ is NOT conjugate to $\beta _2$.

\textup{(3)} Two flows preserve the Lebesgue measure.

\bigskip

As a probability measured space, we have
\[ \prod _{n\in \mathbf{Z}}\{ 0,1 \} ^n =(\prod _{n:\mathrm{odd}}\{ 0,1 \} ^n) \times (\prod _{n:\mathrm{even}}\{ 0,1 \} ^n )\cong \mathbf{T}^2. \] 
By this identification, we set 
\[ \beta :=( \bigotimes _{m\in \mathbf{Z}}(\beta _1\otimes \beta _2))\otimes \mathrm{id}: \mathbf{R}\curvearrowright L^\infty (X\times [0,1)). \]
Let $\theta $ be a flow on $L^\infty(X\times [0,1))$ defined by $T$ and the ceiling function $r=1$. Let $\rho $ be an automorphism of $L^\infty (X\times [0,1))$ defined by $S\times \mathrm{id}$. Then we have the following.

\bigskip

\textup{(1)} The flow $\theta $ commutes with both $\rho $ and $\beta $.

\textup{(2)} The action $\rho $ does not commute with $\beta $.

\textup{(3)} The flow $\theta $ is ergodic.

\bigskip

Now, we construct a pair of flows which are mutually cocycle conjugate but not strongly cocycle conjugate. Let $M$ be an AFD factor of type $\mathrm{III}_0$ with its flow of weights $\{\theta , X\times [0,1)\}$, $\alpha $ be a Rohlin action satisfying $\mathrm{mod}(\alpha )=\beta$ and let  $\sigma $ be an an automorphism of $M$ satisfying $\mathrm{mod} (\sigma )=\rho $. Then we have
\[ \mathrm{mod}(\alpha )=\beta \not =\rho \circ \beta \circ \rho ^{-1}=\mathrm{mod}(\sigma \circ \alpha \circ \sigma ^{-1}).\]
Hence $\alpha $ is cocycle conjugate to $\sigma \circ \alpha \circ \sigma ^{-1}$ but they are not strongly cocycle conjugate.
\end{exam}

\subsection{On Stability}
\label{extend}
In Thoerem 5.9, Izumi \cite{I} has shown that an action of a compact group on any factor of type $\mathrm{III}$ with faithful Connes--Takesaki module is minimal. As well as our main theorem, this theorem means that actions which are ``very outer'' at any non-trivial point are ``globally outer''. He has also shown that for these actions, cocycle conjugacy coincides with conjugacy. This phenomenon also occurs for trace-scaling flows on any factor of type $\mathrm{II}_\infty$. Hence it may be expected that this is true under our assumption. However, this is not the case. 

\begin{theo}
\label{non-stable}
Let $C$ be an abelian von Neumann algebra and $\theta $ be an ergodic flow on $C$. Let $M$ be an AFD factor with its flow of weights $(C, \theta )$. Let $\beta $ be a faithful flow on $C$ which commutes with $\theta$ and fixes a normal faithful semifinite weight $\mu $ of $C$. If the discrete spectrum of $\beta $ is NOT $\mathbf{R}$, then there are two flows $\alpha ^1$, $\alpha ^2$ which satisfies the following two conditions.

\textup{(1)} The Connes--Takesaki modules of $\alpha ^1 $ and $\alpha ^2 $ are $\beta$.

\textup{(2)} The flow $\alpha ^1 $ is NOT conjugate to $\alpha ^2$.
\end{theo}

Before starting the proof, it is important to note that from the above theorem, cocycle conjugacy does not coincide with conjugacy at all. For example, if $\beta $ preserves a normal faithful state of $C$, then its discrete spectrum is countable. Hence it seems that our situation is different from that of actions of compact groups.

\bigskip

We return to the proof of the theorem. In the following, we actually construct these flows. In the following, we denote the AFD factor of type $\mathrm{II}_1$ by $R_0$ and denote the AFD factor of type $\mathrm{II}_\infty $ by $R_{0,1}$.

Let $\Lambda$ be the discrete spectrum of $\beta$ and $\mu $ be a $\beta$-invariant measure. In the rest of this subsection, we assume that $\Lambda $ is not $\mathbf{R}$. Then by the ergodicity of $\theta $ (Note that $\beta $ may not be ergodic), $\Lambda $ is a proper subgroup of $\mathbf{R}$. Hence there are at least two real numbers which do not belong to $\Lambda $. Let $\Gamma _j$ $(j=1,2)$ be two subgroups of $\mathbf{R}$ generated by two elements $\lambda_j, \mu _j$, respectively, satisfying the following conditions.

\[ \Gamma _1 \cup \Lambda \not \subset \langle \Gamma _2,  \Lambda \rangle ,\]

\[ \Gamma _2 \cup \Lambda \not \subset \langle \Gamma _1,  \Lambda \rangle .\]

 Here, $\langle \Gamma _i, \Lambda \rangle $ is the subgroup of $\mathbf{R}$ generated by $\Gamma_i$ and $\Lambda $. Let $\gamma ^j$ ($j=1,2$) be two  ergodic flows on $R_0$ with their discrete spectrum $\Gamma _j$, respectively. Namely, we think of $R_0$ as a weak closure of an irrational rotation algebra $A_s:=\mathrm{C}^*(u,v\mid u,v : \mathrm{unitaries} \   \mathrm{satisfying} \ vu=e^{2\pi is}uv)$ and define flows $\gamma ^j$, $j=1,2$, by the following way. This type of actions is considered by Kawahigashi \cite{Kwh1}.

\[ \gamma ^j_t(u)=e^{i\lambda _j t}u,\  \gamma ^j_t(v)=e^{i\mu _j t}v\]
for $t\in \mathbf{R}$.

Set $\tau :=\mu \otimes \tau _{R_{0,1}}\otimes \tau _{R_0}$. The flow $\theta $ is extended to a $\tau$-scaling flow on $N:=C\otimes R_{0,1}\otimes R_0$ as in equations (1.2) of Sutherland--Takesaki \cite{ST3}. Set $\overline{\alpha ^j}:=\beta \otimes \mathrm{id}_{R_{0,1}}\otimes \gamma ^j$. Then $\overline{\alpha ^j}$ commutes with $\theta$ (See the equation after equation (1.8) of Sutherland--Takesaki \cite{ST3}). Hence the flow $\overline{\alpha ^j}$ is extended to $M:=N\rtimes _\theta \mathbf{R}$ in the following way.

\[ \alpha ^j_t(\lambda ^\theta _s)=\lambda ^\theta _s\]
for $s,t\in \mathbf{R}$. Note that the flow $\theta :\mathbf{R}\curvearrowright N$ is not so ``easy''. However, the flow $\alpha ^j$ is very concrete. Here, we think of $M$ as a von Neumann algebra generated by $N$ and a one parameter unitary group $\{ \lambda _s \} _{s\in \mathbf{R}}$.

In order to show Theorem \ref{non-stable}, for these $\alpha ^j$'s, it is enough to show the following lemma.

\begin{lemm}
\label{two statements}
In the above context, we have the following two statements.

\textup{(1)} The Connes--Takesaki module of $\alpha ^j$ is $\beta $ for each $j=1,2$.

\textup{(2)} For the discrete spectrum of $\alpha ^j$, we have the following inclusion.
\[\Gamma ^j \cup \Lambda \subset \mathrm{Sp}_d(\alpha ^j)\subset \langle \Gamma ^j , \Lambda \rangle .\]
\end{lemm}

From statement (1) of the lemma and our main theorem, it is shown that $ \alpha ^1$ and $\alpha^2$ are mutually cocycle conjugate. On the other hand, from statement (2) of the lemma, it is shown that the discrete spectrum of $\alpha ^1$ and that of $\alpha ^2 $ are different. Hence they are not conjugate.

In order to show this lemma, we first show the following lemma.
 
\begin{lemm}
The weight $\hat{\tau}$ is invariant by $\alpha ^j$.
\end{lemm}

\begin{proof} Set
\[ n_\tau:=\{a\in N \mid \tau (a^*a)<\infty\},\]
\[ K(\mathbf{R},N):=\{ x:\mathbf{R}\to N\mid \mathrm{strongly}^* \ \mathrm{continuous} \ \mathrm{map} \ \mathrm{with} \ \mathrm{compact} \ \mathrm{support}\} ,\]
\[ b_\tau :=\mathrm{span}\{ xa \mid x\in K(\mathbf{R},N), a\in n_\tau \}.\]
For $x\in b_\tau$, set 
\[ \tilde{\pi}(x):=\int_{\mathbf{R}}x_t\lambda _t^\theta \ dt. \] 

In order to show this lemma, it is enough to show the following two statements (For example, see Theorem X.1.17. of Takesaki \cite{T1}).

\bigskip

(1) For $s,t\in \mathbf{R}$, we have $\sigma ^{\hat{\tau }}_t=\alpha ^j _{-s} \circ \sigma ^{\hat{\tau }}_t \circ \alpha ^j_s$.

(2) For $x\in b_\tau$, $s\in \mathbf{R}$, we have 
\[ \hat{\tau}\circ \alpha ^j_s (\tilde{\pi}(x)^*\tilde{\pi}(x))=\hat{\tau}(\tilde{\pi} (x)^*\tilde{\pi}(x)).\]

\bigskip

Statement (1) is trivial because $\alpha ^j$ commutes with $\sigma ^{\hat{\tau}}$. We show statement (2). Notice that 
\begin{align*}
\alpha _s^j(\tilde{\pi}(x)) &=\alpha _s^j (\int _{\mathbf{R}}x_t\lambda _t \ dt) \\
            &=\int _{\mathbf{R}}\overline{\alpha ^j}_s(x_t)\lambda _t \ dt \\
            &=\tilde{\pi}(\overline{\alpha ^j}_s(x)).
\end{align*}
 
Since $\tau $ is invariant by $\overline{\alpha ^j}$, we have
\begin{align*}
\hat{\tau}\circ \alpha ^j _s(\tilde{\pi}(x)^*\tilde{\pi}(x)) &=\hat{\tau}(\tilde{\pi}(\overline{\alpha ^j}_s(x))^*\tilde{\pi}(\overline{\alpha ^j}_s(x))) \\
                                                             &=\tau (\int _{\mathbf{R}} \overline{\alpha ^j}_s (x_t^*x_t) \ dt) \\
                                                             &=\tau (\int _{\mathbf{R}}x_t^*x_t \ dt) \\
                                                             &=\hat{\tau}(\tilde{\pi}(x)^*\tilde{\pi}(x)).
\end{align*}
Thus statement (2) holds.

\end{proof}

By this lemma, the canonical extention $\tilde{\alpha ^j}$ of $\alpha ^j$ is defined by $\tilde{\alpha ^j}_t(\lambda ^\sigma _s)=\lambda ^\sigma _s$ if we think of $\tilde {M}:=M\rtimes _{\sigma ^{\hat{\tau}}}\mathbf{R}$ as a von Neumann algebra generated by $M$ and a one parameter unitary group $\{\lambda ^\sigma _t\}$.

Hence by Lemma 13.3 of Haagerup--St\o rmer \cite{HS}, if we identify $N\rtimes _\theta \mathbf{R} \rtimes _\sigma \mathbf{R}$ with $N \otimes B(H)$ by Takesaki's duality theorem, we have
\[ \tilde{\alpha ^j} \cong \overline{\alpha ^j}\otimes \mathrm{id}.\]
Thus statement (1) of Lemma \ref{two statements} holds.

In the following, we show statement (2) of Lemma \ref{two statements}. We need to show the following lemma.

\begin{lemm}
\label{spect}
We have $\mathrm{Sp}_d\overline{(\alpha ^j})=\mathrm{Sp}_d(\alpha ^j)=\mathrm{Sp}_d(\tilde{\alpha ^j})$.
\end{lemm}
\begin{proof}
The action $\alpha ^j$ is an extension of the action $\overline{\alpha ^j}$, and the action $\tilde{\alpha ^j}$ is an extension of the action $\alpha ^j$. Hence we have $\mathrm{Sp}_d(\overline{\alpha ^j})\subset \mathrm{Sp}_d(\alpha ^j) \subset \mathrm{Sp}_d(\tilde {\alpha ^j})$. We show the implication $\mathrm{Sp}_d(\tilde{\alpha ^j})\subset \mathrm{Sp}_d(\overline{\alpha ^j})$. Note that if we identify $N\rtimes _\theta \mathbf{R} \rtimes _\sigma \mathbf{R}$ with $N\otimes B(H)$ by Takesaki's duality theorem, we have $\tilde {\alpha ^j}=\overline{\alpha ^j} \otimes \mathrm{id}$. Choose $p\in \mathrm{Sp}_d(\tilde{\alpha ^j})$. Then there exists a non-zero element $x\in N\otimes B(H)$ with $\tilde{\alpha ^j}_t(x)=e^{ipt}x$ for $t\in \mathbf{R}$. If we write $x=(x_{kl})_{kl}\in N\otimes B(l^2(\mathbf{N}))$, then there exists $(k,l)$ with $x_{kl} \not =0$. Since we have $\overline{\alpha ^j}_t(x_{kl})=e^{ipt}x_{kl}$, we have $p\in \mathrm{Sp}_d(\overline{\alpha ^j})$. 
\end{proof}

Now, we return to the proof of statement (2) of Lemma \ref{two statements}, which completes the proof of Theorem \ref{non-stable}.

\bigskip

\textit{Proof of Lemma \ref{two statements}.}
The inclusion $\Gamma _j\cup \Lambda \subset \mathrm{Sp}_d(\overline{\alpha ^j})$ is trivial. We show the inclusion $\mathrm{Sp}_d(\overline{\alpha ^j})\subset \langle \Gamma _j, \Lambda \rangle$. If we think of $N=C\otimes R_{0,1}\otimes R_0$ as a subalgebra of $C\otimes B(H)\otimes R_0$, then $\overline{\alpha ^j} $ extends to $\beta \otimes \mathrm{id}_{B(H)}\otimes \gamma ^j$. Hence by the same argument as in Lemma \ref{spect}, we have $\mathrm{Sp}_d(\overline{\alpha ^j})=\mathrm{Sp}_d(\beta \otimes \gamma^j)$.  Choose $p\in \mathrm{Sp}_d(\beta \otimes \gamma ^j)$. Let $x\in C\otimes R_0$ be a non-zero eigenvector for $p\in \mathrm{Sp}_d(\overline{\alpha^j})$. Then $x $ is expanded as
\[ x =\sum _{n,m}c_{n,m} u^nv^m\]
with $c_{n,m} \in C$ ($n,m\in \mathbf{Z}$). Hence we have 
\begin{align*}
\sum _{n,m} e^{ipt}c_{n,m}u^nv^m &= e^{ipt}x \\
                                 &=\beta _t \otimes \gamma ^j _t (x) \\
                                 &=\sum _{n,m} \beta _t (c_{n,m}) e^{i(n\lambda _j + m\mu _j)t}u^nv^m. 
\end{align*}
Since $x \not =0$, there exists $(n,m)$ with $c_{n,m}\not =0$. Hence by the uniqueness of the Fourier expansion, we have 
\[ \beta _t(c_{n,m})=e^{i(p-n\lambda_j -m\mu _j )t} c_{n,m}.\]
Thus $p\in \langle \Gamma ^j, \Lambda \rangle$. 
\qed

\begin{rema}
(1) As shown in Corollary 8.2 of Yamanouchi \cite{Y}, if we further assume that $\alpha ^1$, $\alpha ^2$ and $\beta $ are integrable, then $\alpha ^1$ is conjugate to $\alpha ^2$. In this case, $\beta $ contains the translation  of $\mathbf{R}$ as a direct product component.

\bigskip

(2) Another important difference between flows and actions of compact groups is about extended modular actions. The duals of extended modular flows are important examples of flows with faithful Connes--Takesaki modules (See Theorem 4.20 of Masuda--Tomatsu \cite{MT}). Actions of compact groups with faithful Connes--Takesaki modules are duals of skew products (See Definition 5.6 and Theorem 5.9 of Izumi \cite{I}). However, this is not true for flows by subsection \ref{model} of this paper and Theorem 4.20 of Masuda--Tomatsu \cite{MT}.
\end{rema}

\subsection{On a Characterization of the Rohlin Property}
\label{generalization}
One of the ultimate goals of the study of flows is to completely classify all flows on AFD von Neumann algebras. In order to achieve this, it is important to characterize the Rohlin property by using invariants for flows. A candidate for this characterization is the following conjecture.
\begin{conj}
\label{conj}
\textup{(See Section 8 of Masuda--Tomatsu \cite{MT})} Let $M$ be an AFD von Neumann algebra and let $\alpha $ be a flow on $M$. Let $\tilde{\alpha }:\mathbf{R} \curvearrowright \tilde{M}$ be a canonical extension of $\alpha$. Then the following three conditions are equivalent.

\textup{(1)} The action $\alpha $ has the Rohlin property.

\textup{(2)} We have $\pi _{\tilde{\alpha }}(\tilde{M})' \cap \tilde{M} \rtimes _{\tilde{\alpha }}\mathbf{R}=\pi _{\tilde{\alpha }}(Z(\tilde {M}))$.

\textup{(3)} The action $\alpha $ has full Connes spectrum and is centrally free.
\end{conj}

We will give a partial answer for this conjecture by generalizing Theorem \ref{Rohlin}. We start off by showing the following lemma.

\begin{lemm}
Let $M$ be an AFD factor of type III. Let $\alpha $ be an automorphism of $M$ with trivial Connes--Takesaki module. Then $\alpha $ is centrally outer if and only if $\tilde{\alpha}^\gamma $ is outer for almost every $\gamma \in \Gamma$. Here, $C=L^\infty (\Gamma, \mu)$ is the center of $\tilde{M}$ and $\tilde{\alpha } =\int _\Gamma ^\oplus \tilde{\alpha }^\gamma d\mu (\gamma )$ is the direct integral decomposition. 
\label{ct}
\end{lemm}
\begin{proof}
This is shown by Proposition 5.4 of Haagerup--St\o rmer \cite{HS2} and Theorem 3.4 of Lance \cite{La}.
\end{proof}
In order to state our theorem, we define the following notion.
\begin{defi}
Let $C$ be an abelian von Neumann algebra and let $\beta $ be a flow on $C$. Then $\beta $ is said to be nowhere trivial if for any $e\in \mathrm{Proj}(C^\beta )$, $\beta \mid _{C_e}$ is not $\mathrm{id}_{C_e}$ as a flow.
\end{defi}
The following theorem means that we need not consider Conjecture \ref{conj} for flows on AFD von Neumann algebras of type $\mathrm{III}_0$ anymore.
\begin{theo}
\label{characterization}

\textup{(a)} Let $M$ be a von Neumann algebra of type $\mathrm{III}_0$ and $\alpha $ be a flow on $M$. Assume that $\mathrm{mod}(\alpha )$ is nowhere trivial, then conditions \textup{(1)--(3)} in the above conjecture are all equivalent to the following condition.

\textup{(4)} The action $\alpha $ is centrally free.

\textup{(b)} If conditions \textup{(1)--(3)} are equivalent for flows on the AFD factor of type $\mathrm{II}_\infty $, then these conditions are also equivalent for flows on AFD von Neumann algebras of type $\mathrm{III}_0$. 
\end{theo}
\begin{proof}
\textbf{Step 0.} The implications (1) $\Rightarrow $ (2) and (2) $\Rightarrow $ (3) are shown in Lemma 3.17 and Corollary 4.13 of Masuda--Tomatsu \cite{MT}. The implication (3) $\Rightarrow $ (4) is trivial. 

\textbf{Step 1.}  First, we show (a) and (b) when $M$ is a factor. 

(a) We show the implication (4) $\Rightarrow $ (1). If $\mathrm{mod}(\alpha ):\mathbf{R} \curvearrowright Z(M) $ is faithful, then $\alpha $ satisfies condition (1)  by Theorem \ref{Rohlin}. In the following, we assume that $\mathrm{mod}(\alpha ) $ is not faithful. By the ergodicity of $\theta $, $\mathrm{mod}(\alpha ) $ has a non-trivial period $p\in (0, \infty)$. Since $\theta $ is faithful and commutes with $\mathrm{mod}(\alpha )$, $C^{\mathrm{mod}(\alpha )}$ is not trivial. Hence, the restriction of $\theta $ to $C^{\mathrm{mod}(\alpha )}$ is either free or periodic.

 When the restriction of  $\theta $ to $C^{\mathrm{mod}(\alpha )}$ is free, then the proof goes parallel to Lemma \ref{z}, using Lemma \ref{ct}.

When the restriction of $\theta $ to $C^{\mathrm{mod}(\alpha ) }$ is periodic, then the proof goes parallel to Lemma \ref{z2}.

(b) What remains to do is to reduce the case when $\mathrm{mod}(\alpha ) $ is trivial to Conjecture \ref{conj} for flows on the AFD factor of type$\mathrm{II}_\infty $. This goes parallel to the proof of Lemma \ref{r}.

\textbf{Step 2.} Next, we consider the proof of this theorem for the case when $M$ is not a factor. Decomposing into a direct integral, we may assume that $\alpha $ is centrally ergodic. We need to consider the case when $\alpha \mid_{Z(M)}$ is faithful, the case when $\alpha \mid _{Z(M)}$ has a non-trivial period and the case when $\alpha \mid _{Z(M)}$ is trivial separately. When $\alpha |_{Z(M)}$ is faithful, the implication (4) $\Rightarrow $ (1) follows from Theorem \ref{rp} and Proposition \ref{A.4}. When $\alpha |_{Z(M)}$ has non-trivial period, then the proof is similar to that of Lemma \ref{z}. When $\alpha |_{Z(M)}$ is trivial, then the implication follows from the case when $M$ is a factor.
\end{proof}
\begin{rema}
By the same argument, it is possible to reduce Conjecture \ref{conj} for flows on the AFD factor of type $\mathrm{III}_\lambda $ \textup{(}$0<\lambda <1$\textup{)}, $\mathrm{III}_1$ to Conjecture \ref{conj} for actions of $\mathbf{R}\times \mathbf{Z}$, actions of $\mathbf{R}^2$ on the AFD factor of type  $\mathrm{II}_\infty $, respectively.
\end{rema}

\section{Appendix}
\label{appendix}
In this section, we explain the proof of two statements which are used in the proof of the main theorem.

\bigskip

\subsection{Proof of Theorem \ref{rp}}
For readers who do not have any access to Feldman \cite{F2}, we will explain the outline of the proof of Theorem \ref{rp}.

\bigskip

\textit{Proof of Theorem \ref{rp}.}

The proof consists of two parts. The first is, for any cube $Q$ of $\mathbf{R}^d$, constructing a $Q$-set $F$ with $\mu (QF)>0$. This part is shown by the same argument as in the proof of Lemma of Lind \cite{L} (Note that Wiener's ergodic theorem holds for actions without invariant measures). The second is to show this theorem by using the first part. This is achieved by the same argument as in the proof of Theorem 1 of Feldman--Lind \cite{FL}. In the proof, they show two key statements (Statements (i) and (ii) in p.341 of Feldman--Lind \cite{FL}). We need statements corresponding to them. Let $L$, $N$, $P$ be positive natural numbers. Assume that $P$ is a multiple of $L$. Set
\[ Q_P:=[0,P)^d,\]
\[ S_L(Q_P):=\{ t=(t_1, \cdots , t_d)\in \mathbf{R}^d\mid \frac{P}{L}\leq t_j <P-\frac{P}{L} \ \mathrm{for} \ \mathrm{all} \ j\},\]
\[ B_N(Q_P):=\{ t=(t_1, \cdots , t_d)\in \mathbf{R}^d\mid -N \leq t_j <P+N \ \mathrm{for} \ \mathrm{all} \ j\}\setminus Q_P,\]
\[ C_{P/L}:=\{ n=(n_1, \cdots , n_d)\in \mathbf{Z}^d \mid 0\leq n_j <\frac{P}{L} \ \mathrm{for} \ \mathrm{all} \ j\}.\]
The corresponding statements are the following.

\bigskip

(i)' Let $\eta >0$ be a positive number. Then for any sufficiently large even integer $M$, any integer $L$, any multiple $P=NLM$ of $LM$ and any $Q_P$-set $F$, we have 
\[ \mu (B_{2N}(Q_P)(tF))<\eta \]
for over $9/10$ of the elements $t$ of $C_{P/L}$.

\bigskip

(ii)' Let $\xi>0$ be a positive number. Then for any sufficiently large integer $L$, any integer $M$, any multiple $P=NLM$ of $LM$ by a multiple $N$ of $L$ and any $Q_P$-set $F$, we have
\[ \mu (S_L(Q_N)(NC_{P/L})(tF))>\mu (Q_P(tF))-\xi \]
for over $9/10$ of the elements $t$ of $C_{P/L}$.

\bigskip

The other parts of the proof is the same as that of Theorem 1 of Feldman--Lind \cite{FL}.    
\qed

\bigskip

We may assume that $X$ is a compact metric space and the map $T:\mathbf{R}^d\times X\to X$ is continuous.
\begin{lemm}
\label{A.3}
In the context of Theorem \ref{rp}, the set $F$ can be chosen to be a Borel subset of $X$.
\end{lemm}
\begin{proof}
This follows from the proof of Lemma of Lind \cite{L}. By removing a null set, we may assume that the set $D$ in p.181 of Lind \cite{L} is a Borel subset of $X$. Then the set 
\[\{ (t,x) \in \mathbf{R}^n \times X \mid T_t(x) \in D \}\]
is a Borel subset of $Q\times X$. Hence by Fubini's theorem, the map $\psi _j^{\pm }(x)$ in p.181 of Lind is Borel measurable. Thus the set $F$ can be chosen to be a Borel subset.
\end{proof}
\begin{prop}
\label{A.4}
In the context of Theorem \ref{rp}, the map 
\[ Q\times F \ni (t,x) \mapsto T_t(x) \in T_QF\]
is a Borel isomorphism.
\end{prop}
\begin{proof}
By Lemma \ref{A.3}, if $C\subset T_QF $ is a Borel subset, then $C$ is also Borel in $X$. Hence the map $Q\times F \ni (t,x)\mapsto T_t(x) \in T_QF$ is a Borel bijection. Hence by Corollary A.10 of Takesaki \cite{T}, this map is a Borel isomorphism.
\end{proof}
\begin{prop}
\label{A.5}
In the context of Theorem \ref{rp}, if $\mathbf{R}^d=\mathbf{R}$, then the map
\[ Q\times F \ni (t,x) \mapsto T_t(x) \in T_QF\]
is non-singular.
\end{prop}
\begin{proof}
This is based on Lemma 3.1 of Kubo \cite{K}. The action $T$ of $\mathbf{R}$ on $X$ induces an action $\tilde{T}$ of $\mathbf{R}$ on $T_Q(F)$. Then $\tilde{T}$ defines an action $S$ of $\mathbf{Z}$ on $F$. Then $(F, \nu)$, $S$ and $(T_QF,\mu)$ satisfy the assumptions of Lemma 3.1 of Kubo \cite{K}.
\end{proof}
\subsection{On a Measurability of a Certain Map}
In the proof of Lemma \ref{10}, we use the fact that a map from a measured space to the automorphism group of a von Neumann algebra is measurable (See Fact between Lemma \ref{11} and Lemma \ref{13}). Probably it is well-known for specialists. However, we could not find appropriate references. Hence, we present the proof here.
\begin{prop} If we identify $N$ with $N_0 \otimes L^\infty ([0,p))$, the map $[0, p)^2\ni(t, \gamma ) \mapsto \theta _{t, \gamma }\in \mathrm{Aut}(N_0)$ is Lebesgue measurable. 
\end{prop}

\begin{proof}  By Lusin's theorem, it is enough to show that the map $[0,p)^2 \ni (t, \gamma ) \mapsto \phi _0 \circ \theta _{t, \gamma } \in (N_0)_*$ is Lebesgue measurable for $\phi _0 \in (N_0)_*$. We identify $N_*$ with $L^1 _{(N_0)_*}([0,p))$ and set $\phi :=\phi _0\otimes \mathrm{id}$. Since the map $s\mapsto \phi \circ \theta _s \in L^1_{(N_0)_*}([0,p))$ is continuous, for any $\epsilon >0$, there exists a positive number $\delta $ such that 

\begin{equation}
 \| \phi \circ \theta _s -\phi \| <\epsilon ^2 
\end{equation}
for $| s|<\delta$. Take a partition $0=s_0<s_1< \cdots <s_n=p$ so that $| s_i-s_{i+1}| <\delta$. For each $i=0, \cdots , n$, the map $[0,p)\ni \gamma \mapsto (\phi \circ \theta _{s_i})_\gamma$ is Lebesgue measurable and integrable. Hence it is possible to approximate $\phi \circ \theta _{s_i}$ by Borel simple step functions, that is, for each $i$, there exists a compact subset $K_i$ of $[0,p)$ which satisfies the following conditions.

\bigskip

(2) We have $\mu (K_i)>p-\epsilon$.

(3) There exist a Borel partition $\{ I_j\}$ of $K_i$ and $\phi _{i,j}\in( N_0)_*$ such that 
\[ \| (\phi \circ \theta _{s_i})_\gamma -\sum _j \phi _{i,j}\chi _{I_j}(\gamma ) \| <\epsilon  \]
for $ \gamma \in K_i$. 

\bigskip

 Set 
\[ \psi _{t, \gamma }:=\sum _{i,j}\phi _{i,j}\chi _{[s_i,s_{i+1})}(t)\chi _{I_j}(\gamma ). \]
for each $(t, \gamma ) \in [0,p)^2$. For each $s\in [s_i,s_{i+1})$, set
\[ K_s:=\{ \gamma \in [0,p)\mid \| (\phi \circ \theta _s) _\gamma -(\phi \circ \theta _{s_i})_\gamma \| <\epsilon \} . \] 
Then by the above inequality (1), we have $\mu (K_s)>p-\epsilon $. For $\gamma \in K_s\cap K_i$, we have 
\[ \| (\phi \circ \theta _s)_\gamma -\psi _{s, \gamma } \| <2\epsilon .\]
Set
\[ K:=\{ (s, \gamma ) \in [0,p)^2 \mid \| (\phi \circ \theta _s)_\gamma -\psi _{s, \gamma } \| <2\epsilon \} .\]
Then we have $\mu (K) >p(p-2\epsilon )$. Hence $(s, \gamma ) \mapsto (\phi \circ \theta _s)_\gamma $ is well-approximated by simple step Borel functions in measure convergence. Hence this is Lebesgue measurable.
\end{proof}

\textbf{Acknowledgements.}
The author is thankful to Professor Toshihiko Masuda for pointing out a mistake in the proof of Lemma \ref{7} in an early draft and giving him useful advice about Theorem \ref{rp}. He also thanks to Professor Kawahigashi, who is his adviser, for encouragement on this work and for giving him useful advice about the presentation of the paper. He also  thanks to Professor Izumi, Professor Tomatsu and Professor Ueda for giving him useful advice about subsection \ref{extend}. He is supported by the Program for Leading Graduate
Schools, MEXT, Japan.


\begin{thebibliography}{99}
\bibitem{C}
A. Connes, Almost periodic states and factors of type $\mathrm{III}_1$  . J. Funct. Anal. \textbf{16} (1974), 415--445.
\bibitem{C2}
A. Connes, Factors of type $\mathrm{III}_1$, property $L'_\lambda$ and closure of inner automorphisms, J. Operator Theory \textbf{14}, no. 1 (1985), 189--211.
\bibitem{C3}
A. Connes, Outer conjugacy classes of automorphisms of factors, Ann. Sci. \'E cole Norm. Sup. (4) \textbf{8} (1975), no. 3, 383--419. 
\bibitem{CT}
A. Connes and M. Takesaki, The flow of weights on factors of type III, Tohoku Math. J. (2) \textbf{29}, no. 4 (1977), 473--575.
\bibitem{F}
J. Feldman, Changing orbit equivalences of $\mathbf{R}^d$ actions, $d\geq 2$, to be $C^\infty $ on orbits, Internat. J. Math. \textbf{2}, no. 4 (1990), 409--427.
\bibitem{F2} 
J. Feldman, $\mathrm{C}^\infty$ Orbit Equivalence of Flows, xeroxed lecture notes, privately circulated.
\bibitem{FHM}
J. Feldman, P. Hahn and C. C. Moore, Orbit structure and countable sections for actions of continuous groups, Adv. Math. \textbf{28}, no. 3 (1978), 186--230.
\bibitem{FL}
J. Feldman and D. Lind, Hyperfiniteness and the Halmos-Rohlin theorem for nonsingular Abelian actions. Proc. Amer. Math. Soc.  \textbf{55}, no. 2 (1976), 339--344.  
\bibitem{I}
M. Izumi, Canonical extension of endomorphisms of type III factors. Amer. J. Math.  125,  no. 1 (2003), 1--56.
\bibitem{J}
V. F. R. Jones, Actions of finite groups on the hyperfinite type $\mathrm{II}_1$ factor, Mem. Amer. Math. Soc. \textbf{237} (1980).
\bibitem{KtST}
Y. Katayama, C. E. Sutherland and M. Takesaki, The characteristic square of a factor and the cocycle conjugacy of discrete group actions on factors, Invent. Math. \textbf{132}, no. 2, (1998), 331--380. 
\bibitem{Kwh1}Y. Kawahigashi, One-parameter automorphism groups of the injective $\mathrm{II}_1$ factor arising from the irrational rotation C$^*$-algebra, Amer. J. Math. \textbf{112} no. 4 (1990), 499--523.
\bibitem{Kwh2}Y. Kawahigashi, One-parameter automorphism groups of the injective $\mathrm{II}_1$ factor with Connes spectrum zero,
Canad. J. Math. \textbf{43} (1991), 108--118.
\bibitem{KwhST}
Y. Kawahigashi, C. E. Sutherland, and M. Takesaki, The structure of the automorphism group of an injective factor and the cocycle conjugacy of discrete abelian group actions, Acta Math. \textbf{169} (1992), 105--130.
\bibitem{Kwm}
K. Kawamuro, A Rohlin property for one-parameter automorphism groups of the hyperfinite $\mathrm{II}_1$ factor, Publ. Res. Inst. Math. Sci. \textbf{36} (2000), no. 5, 641--657.
\bibitem{Ksm}
A. Kishimoto, A Rohlin property for one-parameter automorphism groups, Comm. Math. Phys. \textbf{179} (1996), no. 3, 599--622.
\bibitem{K}
I. Kubo, Quasi-flows, Nagoya Math. J. \textbf{35} (1969), 1--30. 
\bibitem{H}
U. Haagerup, Connes bizentralizer problem and uniqueness of the injective factor of type $\mathrm{III}_1$, Acta Math. \textbf{158} (1987), 95--148.
\bibitem{HS}
U. Haagerup and E. St\o rmer, Equivalence of normal states on von Neumann algebras and the flow of weights. Adv. Math.  \textbf{83}  (1990),  no. 2, 180--262.
\bibitem{HS2}
U. Haagerup and E. St\o rmer, Pointwise inner automorphisms of von Neumann algebras. With an appendix by Colin Sutherland, J. Funct. Anal. \textbf{92} (1990), 177--201
\bibitem{La}
C. Lance, Direct integrals of left Hilbert algebras, Math. Ann. \textbf{216} (1975), 11--28.
\bibitem{L}
D. Lind, Locally compact measure preserving flows, Adv. Math. \textbf{15} (1975), 175--193.
\bibitem{M}
T. Masuda, Unified approach to classification of actions of discrete amenable groups on injective factors, J. Reine Anguew. Math. \textbf{683} (2013), 1--47. 
\bibitem{MT2}
T. Masuda and R. Tomatsu, Classification of actions of discrete Kac algebras on injective factors, Preprint (2013), arXiv:1306.5046. 
\bibitem{MT}
T. Masuda and R. Tomatsu, Rohlin flows on von Neumann algebras, Preprint (2012), arXiv:1206.0955.
\bibitem{O}
A. Ocneanu, Actions of discrete amenable groups on von Neumann algebras, Lecture Notes in Mathematics, \textbf{1138}. Springer-Verlag, Berlin, 1985. iv+115 pp.
\bibitem{S}
K. Shimada, Rohlin flows on Amalgamated Free Product Factors, to appear in Internat. Math. Res. Notices.
\bibitem{ST2}
C. E.  Sutherland and M. Takesaki, Actions of discrete amenable groups on injective factors of type $\mathrm{III}_\lambda$, $\lambda \not=0$, Pacific J. Math. \textbf{137} (1989), 405--444.
\bibitem{ST3}
C. E.  Sutherland and M. Takesaki, Right inverse of the module of approximately finite-dimensional factors of type III and approximately finite ergodic principal measured groupoids.  Operator algebras and their applications, II (Waterloo, ON, 1994/1995),  149--159, Fields Inst. Commun., 20, Amer. Math. Soc. Providence, RI, 1998.
\bibitem{T}
M. Takesaki, Theory of operator algebras. I, Reprint of the first (1979) edition. Encyclopaedia of Mathematical Sciences, 124. Operator Algebras and Non-commutative Geometry, 5. Springer-Verlag, Berlin, 2002. xx+415 pp.  
\bibitem{T1} 
M. Takesaki, Theory of operator algebras. II, Encyclopedia of Mathematical Sciences, 125. Operator Algebras and Non-commutative Geometry, 6. Springer-Verlag, Berlin, 2003. xxii+518 pp.
\bibitem{Y}
T. Yamanouchi, One-cocycles on smooth flows of weights and extended modular coactions. Ergodic Theory Dynam. Systems  \textbf{27},  no. 1 (2007), 285--318.
\end{thebibliography}
\end{document}